\documentclass{article}
\usepackage{amsthm}
\usepackage{amssymb}
\usepackage{amsmath}
\usepackage{epsfig}
\usepackage{colonequals}
\usepackage{enumerate}
\usepackage{hyperref}

\newcommand{\claim}[2]{\begin{equation}\mbox{\parbox{\linewidth}{{\em #2}}}\label{#1}\end{equation}}
\newtheorem{theorem}{Theorem}[section]
\newtheorem{corollary}[theorem]{Corollary}
\newtheorem{lemma}[theorem]{Lemma}

\begin{document}
\title{Fine structure of $4$-critical triangle-free graphs I.  Planar graphs with two triangles and $3$-colorability of chains}
\author{%
     Zden\v{e}k Dvo\v{r}\'ak\thanks{Computer Science Institute (CSI) of Charles University,
           Malostransk{\'e} n{\'a}m{\v e}st{\'\i} 25, 118 00 Prague, 
           Czech Republic. E-mail: \protect\href{mailto:rakdver@iuuk.mff.cuni.cz}{\protect\nolinkurl{rakdver@iuuk.mff.cuni.cz}}.
           Supported by project 14-19503S (Graph coloring and structure) of Czech Science Foundation.}
\and	   
Bernard Lidick\'y\thanks{
Iowa State University, Ames IA, USA. E-mail:
\protect\href{mailto:lidicky@iasate.edu}{\protect\nolinkurl{lidicky@iastate.edu}}.
Supported by NSF grant DMS-1266016 and DMS-1600390.}
}
\date{\today}
\maketitle
\begin{abstract}
Aksenov proved that in a planar graph $G$ with at most one triangle, every precoloring of a $4$-cycle can be extended to
a $3$-coloring of $G$.  We give an exact characterization of planar graphs with two triangles in which some precoloring of
a $4$-cycle does not extend.  We apply this characterization to solve the precoloring extension problem from two $4$-cycles in a triangle-free
planar graph in the case that the precolored $4$-cycles are separated by many disjoint $4$-cycles.  The latter result is
used in followup papers to give detailed information about the structure of $4$-critical triangle-free graphs embedded in a fixed surface.
\end{abstract}

\section{Introduction}

The interest in the $3$-coloring properties of planar graphs was started by a celebrated theorem of Gr\"otzsch~\cite{grotzsch1959},
who proved that every planar triangle-free graph is $3$-colorable.  While in general, deciding $3$-colorablity
of a planar graph is an NP-complete problem~\cite{garey1979computers}, there are many other sufficient conditions guaranteeing
$3$-colorability, see e.g. the survey of Montassier~\cite{montasweb}.

For a long time, the question of the complexity of deciding whether a triangle-free graph embedded in a fixed surface (other than
the sphere) is $3$-colorable was open.  The question was resolved for the projective plane by the result of Gimbel and Thomassen~\cite{gimbel},
and in a far reaching generalization, Dvo\v{r}\'ak, Kr\'al' and Thomas~\cite{trfree7} proved that there exists a linear-time
algorithm for this problem for any fixed surface, even if a bounded number of vertices have prescribed colors.  In order to design their algorithm,
they show in~\cite{trfree6} that every triangle-free graph embedded in a fixed surface exhibits a special structure that determines its $3$-coloring properties.

Before we describe their structural result, let us recall some definitions.  A \emph{surface} is a two-dimensional manifold, possibly with
boundary.  By the surface classification theorem, each surface can be (up to homeomorphism) obtained from the sphere by adding a finite
number of handles and crosscaps, and in the case of surfaces with boundary, by drilling a finite number of holes.  
The \emph{disk} is the sphere with a hole and the \emph{cylinder} is the sphere with two holes.
An \emph{embedding} of a graph $G$ in a surface $\Sigma$ is a function $\alpha$ that maps vertices of $G$ to
distinct points in $\Sigma$ and edges of $G$ to simple curves in $\Sigma$ intersecting only in their endpoints,
such that for each $uv\in E(G)$, $\alpha(u)$ and $\alpha(v)$
are the endpoints of the curve $\alpha(uv)$ and no other vertices are mapped to points in this curve.  Throughout the paper,
graphs will generally be embedded in some surface; we usually keep the embedding implicit and we use terms such as vertex and edge
to refer to both the elements of the graph and the points or curves in the surface that represent them.
A \emph{face} of a graph $G$ embedded in $\Sigma$ is a connected component of the surface after removing the points and curves of
the embedding of $G$; in particular, if $\Sigma$ is a surface with boundary and a cycle $C$ in $G$ traces a component of the boundary,
then $C$ does not necessarily bound a face.  A closed walk in $G$ is \emph{contractible} if the closed curve tracing this walk in the surface
is null-homotopic.  A contractible cycle bounds a disk in $\Sigma$, unique unless $\Sigma$ is the sphere; the cycle is \emph{facial} if
the interior of such a disk is a face of $G$.

We are now ready to state the structural result of Dvo\v{r}\'ak et al.~\cite{trfree7}.
Let $G$ be a graph embedded in a surface $\Sigma$ so that every contractible $4$-cycle is facial,
and suppose $G$ intersects the boundary of $\Sigma$ in a set $X$ of $k$ vertices (which we view as precolored).
Then $\Sigma$ can be cut in to a bounded number of pieces along curves tracing closed walks in $G$, so that
\begin{itemize}
\item the total number of vertices contained in the boundaries of the pieces is bounded by a constant depending only on $\Sigma$ and $k$,
\item each piece $\Pi$ and the subgraph $H$ of $G$ drawn in $\Pi$ satisfies one of the following:
\begin{enumerate}
\item every $3$-coloring of the vertices of $H$ contained in the boundary of $\Pi$ extends to a $3$-coloring of $H$; or,
\item a $3$-coloring of the vertices of $H$ contained in the boundary of $\Pi$ extends to a $3$-coloring of $H$ if and only if it satisfies a specific condition
(``the winding number constraint''); or,
\item $\Pi$ is homeomorphic to the cylinder (whose boundary consists of two cycles in $H$ of arbitrary length) and all faces of $H$ have length $4$; or,
\item $\Pi$ is homeomorphic to the cylinder whose boundary consists of two cycles in $H$ of length $4$.
\end{enumerate}
\end{itemize}
Thus, to determine whether a precoloring of $X$ extends to a $3$-coloring $G$, we can try all the (constantly many) extensions to $3$-colorings of the boundary vertices
of the pieces, and then test whether one of them extends to all the pieces.  In the first two possibilities for the pieces of the structure we have
a complete information about which colorings of the boundaries of the pieces extend, and the last part is thus trivial.
However, in the last two cases the information is much more limited.  While this is sufficient for the purposes
of the algorithm of~\cite{trfree7}, it would be preferable to have a more detailed structural theorem
where the $3$-coloring properties of all the pieces are known.  This is the main aim of this series of papers.

In this paper we focus on the last subcase of a graph embedded in the cylinder with boundary cycles of length $4$.
Note that if such a graph contains only a bounded number of separating $4$-cycles, then we can further cut the
surface along them and by using the ideas of~\cite{trfree6}, we can subdivide the pieces to a bounded number of subpieces
satisfying the conditions of one of the first two well-understood cases of the structure theorem.  Hence, it is interesting to study
the graphs in cylinder with many separating $4$-cycles, and this is the topic of this paper.

Let us now give a few definitions enabling us to state the main result more precisely.
In this paper, we generally consider graphs embedded in the sphere, the disk, or the cylinder.
Suppose that $G$ is a graph embedded in a surface $\Sigma$ with boundary and consider a component $\Theta$
of the boundary ($\Theta$ is a simple closed curve bounding a hole in the surface).  Let $\Sigma_\Theta$ denote
the surface obtained from $\Sigma$ by closing the hole, i.e., by identifying $\Theta$ with the boundary of an open disk $\Lambda_\Theta$
disjoint from $\Sigma$.  We say that $\Theta$ is \emph{surrounded} in $G$ if the embedding of $G$ in $\Sigma_\Theta$
has a face homeomorphic to an open disk containing $\Lambda_\Theta$ and bounded by a cycle $R$.
Equivalently, the part of the surface $\Sigma$ between $R$ and $\Theta$ intersects the drawing of $G$ exactly in $R$
and all non-contractible simple curves contained in this part are homotopically equivalent to the closed curve tracing $R$.
In that case, we say that $R$ is the \emph{ring} surrounding the hole.

From now on, we always assume that if a graph $G$ is embedded in a surface with boundary,
then all the holes of the surface are surrounded.
Note that the ring may or may not trace the boundary of the hole it surrounds, and in particular
the rings surrounding different holes do not have to be disjoint (or even distinct, in case that $G$ is just a cycle).
A face $f$ of $G$ is a \emph{non-ring} face if $f$ is not contained in any of the parts of the surface between the holes and the rings that
surround them.

We construct a sequence of graphs $T_1$, $T_2$, \ldots, which we call \emph{Thomas-Walls graphs} (Thomas and Walls~\cite{tw-klein} proved that they are exactly
the $4$-critical graphs that can be drawn in the Klein bottle without contractible cycles of length at most $4$).
Let $T_1$ be equal to $K_4$.  For $n\ge 1$, let $u_1u_3$ be any edge of $T_n$ that belongs to two triangles and let $T_{n+1}$ be obtained from $T_n-u_1u_3$
by adding vertices $x$, $y$ and $z$ and edges $u_1x$, $u_3y$, $u_3z$, $xy$, $xz$, and $yz$.  The first few graphs of this sequence are drawn in Figure~\ref{fig-thomaswalls}.
\begin{figure}
\begin{center}
\includegraphics[scale=0.8]{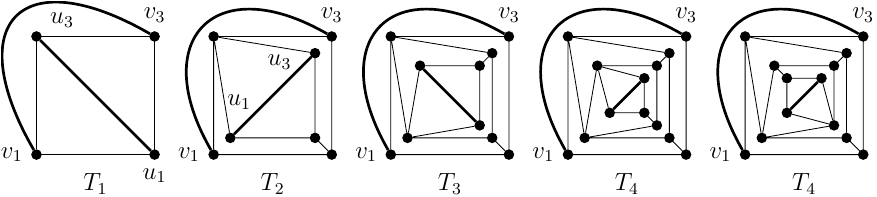}
\end{center}
\caption{Some Thomas-Walls graphs (with two different drawings of $T_4$).}\label{fig-thomaswalls}
\end{figure}
For $n\ge 2$, note that $T_n$ contains unique $4$-cycles $C_1=u_1u_2u_3u_4$ and $C_2=v_1v_2v_3v_4$ such that $u_1u_3,v_1v_3\in E(G)$.
Let $T'_n=T_n-\{u_1u_3,v_1v_3\}$.  We also define $T'_1$ to be a $4$-cycle $C_1=C_2=u_1v_1u_3v_3$.
We call the graphs $T'_1$, $T'_2$, \ldots \emph{reduced Thomas-Walls graphs}, and we say that $u_1u_3$ and $v_1v_3$ are their \emph{interface pairs}.
Note that $T'_n$ has an embedding in the cylinder with rings $C_1$ and $C_2$.  

A \emph{patch} is a graph $F$ drawn in the disk with ring $C$ of length $6$ that traces the boundary of the disk,
such that $C$ is an induced cycle in $F$, every face of $F$ has length $4$, and every $4$-cycle in $F$ is facial.
Let $G$ be a graph embedded in the sphere, possibly with holes.
Let $G'$ be any graph which can be obtained from $G$ as follows.
Let $S$ be an independent set in $G$ such that every vertex of $S$ has degree $3$.  For each vertex $v\in S$
with neighbors $x$, $y$ and $z$, remove $v$, add new vertices $a$, $b$ and $c$ and a $6$-cycle
$C=xaybzc$ (where $a$, $b$, and $c$ are drawn very close to the original location of $v$
and the edges of $C$ are drawn very close to the curves representing the edges $vx$, $vy$, and $vz$),
and draw any patch with ring $C$ in the disk bounded by $C$.
We say that any such graph $G'$ is obtained from $G$ by \emph{patching}.
This operation was introduced by Borodin et al.~\cite{4c4t} in the context of describing planar $4$-critical graphs with exactly $4$ triangles.

Consider a reduced Thomas-Walls graph $G=T'_n$ for some $n\ge 1$, with interface pairs $u_1u_3$ and $v_1v_3$.
A \emph{patched Thomas-Walls graph} is any graph obtained from such a graph $G$ by patching, and $u_1u_3$ and $v_1v_3$ are
its interface pairs (note that $u_1$, $u_3$, $v_1$, and $v_3$ have degree two in $G$, and thus they are not affected by patching).

Let $G$ be a graph embedded in the sphere with $n$ holes ($n\in\{1,2\}$), with rings $C_i=x_iy_iz_iw_i$
of length $4$, for $1\le i\le n$.  Let $y'_i$ be either a new vertex or $y_i$, and let $w'_i$ be either a new vertex
or $w_i$.  Let $G'$ be obtained from $G$ by adding $4$-cycles $x_iy'_iz_iw'_i$ forming the new rings.  We say that
$G'$ is obtained by \emph{framing on pairs $x_1z_1$, \ldots, $x_nz_n$}.

Let $G$ be a graph embedded in the cylinder with rings $C_1$ and $C_2$ of length three, such that every non-ring face of $G$ has length $4$.
We say that such a graph $G$ is a \emph{$3,3$-quadrangulation}.
Let $G'$ be obtained from $G$ by subdividing at most one edge in each of $C_1$ and $C_2$.  We say that such a graph $G'$ is a \emph{near $3,3$-quadrangulation}.

We say that a graph $G$ embedded in the cylinder is \emph{tame} if $G$ contains no
contractible triangles, and all triangles of $G$ are pairwise vertex-disjoint.
Let $G$ be a tame graph embedded in the cylinder with rings of length at most $4$.
We say that $G$ is a \emph{chain} of graphs $G_1$, \ldots, $G_n$, if there exist non-contractible $(\le\!4)$-cycles $C_0$, \ldots, $C_n$
in $G$ such that 
\begin{itemize}
\item the cycles are pairwise vertex-disjoint except that for $(i,j) \in \{(0,1),(n,n-1)\}$,  $C_i$ can intersect $C_j$ if
$C_i$ is a 4-cycle and $C_j$ is a triangle,
\item for $0\le i<j<k\le n$ the cycle $C_j$ separates $C_i$ from $C_k$,
\item the cycles $C_0$ and $C_n$ are the rings of $G$,
\item every triangle of $G$ is equal to one of $C_0$, \ldots, $C_n$, and
\item for $1\le i\le n$, the subgraph of $G$ drawn between $C_{i-1}$ and $C_i$ is isomorphic to $G_i$.
\end{itemize}
We say that $C_0$, \ldots, $C_n$ are the \emph{cutting cycles} of the chain.
The main result of this paper is the following.

\begin{theorem}\label{thm-many}
There exists an integer $c\ge 0$ such that the following holds.
Let $G$ be a tame graph embedded in the cylinder with rings $C_1$ and $C_2$ of length at most $4$.
If $G$ is a chain of at least $c$ graphs, then
\begin{itemize}
\item every precoloring of $C_1\cup C_2$ extends to a $3$-coloring of $G$, or
\item $G$ contains a subgraph $H$ obtained from a patched Thomas-Walls graph by framing on its interface pairs, with rings $C_1$ and $C_2$, or
\item $G$ contains a near $3,3$-quadrangulation $H$ with rings $C_1$ and $C_2$ as a subgraph.
\end{itemize}
\end{theorem}

Aksenov~\cite{aksenov} proved the following strengthening of Gr\"otzsch's theorem (fixing a previous flawed proof of this
fact by Gr{\"u}nbaum~\cite{grunbaum}).

\begin{theorem}[Aksenov~\cite{aksenov}]\label{thm-3t}
Every planar graph with at most $3$ triangles is $3$-colorable.
\end{theorem}

In the course of the proof, Aksenov also established another interesting fact.

\begin{theorem}[Aksenov~\cite{aksenov}]\label{thm-onetri}
Let $G$ be a graph drawn in the cylinder with a ring $C$ of length at most $4$.
If every triangle in $G$ is non-contractible, then every precoloring of $C$ extends to a $3$-coloring of $G$.
\end{theorem}

As a part of the proof of Theorem~\ref{thm-many}, we need to prove a strengthening of  Theorem~\ref{thm-onetri} and describe
the $3$-coloring properties of graphs embedded in the disk with a ring $C$ of size $4$ and with exactly two triangles.
To state this result (Theorem~\ref{thm-crtri} below) of independent interest, it is convenient to introduce the notion
of a critical graph.

Let $C$ be the union of the rings of a graph $G$ embedded in a surface
($C$ is empty when $G$ is embedded in a surface without boundary).
By a \emph{precoloring} of $C$, we mean any proper $3$-coloring of $C$.
We say that $G$ is \emph{critical} if $G\neq C$ and for every proper subgraph $G'$ of $G$ such that $C\subseteq G'$,
there exists a precoloring of $C$ that extends to a $3$-coloring of $G'$, but not to a $3$-coloring of $G$; that is, removing any
edge or vertex not belonging to $C$ affects the set of precolorings of $C$ that extend to a $3$-coloring of the graph.

In particular, consider any graph $H$ embedded in a surface, let $C$ be the union of its rings, and let $G$ be an inclusionwise-minimal
subgraph of $H$ such that $C\subseteq G$ and every precoloring of $C$ that extends to a $3$-coloring of $G$ also extends to
a $3$-coloring of $H$.  Then either $G=C$ or $G$ is critical, and $G$ carries all the information regarding which precolorings of $C$ extend to $H$;
we say that $G$ is a \emph{critical skeleton} of $H$.
Consequently, it suffices to consider the properties of critical graphs, and we will do so in
Theorem~\ref{thm-crtri} as well as in many of the further results.

We need another construction related to Thomas-Walls graphs.  For $n\ge 1$, let $G=T'_n$ be a reduced Thomas-Walls graph
with rings $C_1=u_1u_2u_3u_4$ and $C_2=v_1v_2v_3v_4$ and interface pairs $u_1u_3$ and $v_1v_3$.
Consider the embedding of $T'_n$ in the disk, obtained by closing the hole in the face bounded by $C_1$.
Let $G'$ be a graph obtained from $G$ by either adding the edge $u_1u_3$, or the subgraph depicted in Figure~\ref{fig-havel}(a) (this graph is
often called \emph{Havel's quasiedge}, since Havel~\cite{conj-havel} used it to disprove a conjecture by Gr\"unbaum that every planar graph without intersecting triangles is 3-colorable).
A \emph{patched Havel-Thomas-Walls graph} is any graph obtained from such a graph $G'$ by patching, and $v_1v_3$ is its \emph{interface pair}.

\begin{figure}
\begin{center}
\includegraphics[page=1]{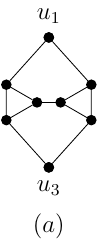}
\includegraphics[page=2]{fig-havel}
\includegraphics[page=3]{fig-havel}
\end{center}
\caption{Graphs related to Theorem~\ref{thm-crtri}.}\label{fig-havel}
\end{figure}

Let a \emph{tent} be a graph embedded in the disk with the ring $v_1v_2v_3v_4$, containing vertices $z_1$
adjacent to $v_1$ and $v_2$, and $z_2$ adjacent to $v_3$ and $v_4$, such that all faces other than $v_1v_2z_1$ and $v_3v_4z_2$
have length $4$, and $v_1v_3,v_2v_4\not\in E(G)$, see Figure~\ref{fig-havel}(c).

\begin{theorem}\label{thm-crtri}
Let $G$ be a graph embedded in the disk with at most two triangles and with the ring of length $4$.
If $G$ is critical, then $G$ is either a tent, or obtained from a patched Havel-Thomas-Walls graph by framing on its interface pair.
\end{theorem}

It is important to note that the precolorings of the rings which extend to graphs appearing in the conclusions of
Theorems~\ref{thm-many} and \ref{thm-crtri} can be precisely described, as we show in Section~\ref{sec-htw}
(Lemma~\ref{lemma-col-tw} for framed patched Thomas-Walls graphs, Lemma~\ref{lemma-col-33} for near $3,3$-quadrangulations,
Corollary~\ref{cor-col-pt} for tents, and Lemma~\ref{lemma-col-htw} for framed patched Havel-Thomas-Walls graphs).
In particular, in the structure theorem we aim for, it is satisfactory to have pieces that are patched reduced Thomas-Walls graphs.

Furthermore, in Theorem~\ref{thm-many}, all non-ring faces $f$ of the graphs $H$ are bounded by contractible $(\le\!5)$-cycles,
and the subgraph $G[f]$ of $G$ drawn in the closure of $f$ contains no triangles.
Consequently, every precoloring of the cycle bounding $f$ extends to a $3$-coloring of $G[f]$ (see e.g. Lemma~\ref{lemma-split} below).
We conclude that a precoloring of the rings of $G$ extends to a $3$-coloring of $G$ if and only if it extends to a $3$-coloring of $H$.

The rest of the paper is structured as follows. In Section~\ref{sec-htw}, we describe possible colorings of the graphs appearing in Theorems~\ref{thm-many} and \ref{thm-crtri}.
Section~\ref{sec-bc} is devoted to examining a chain $G$ of graphs, many of which are not quadrangulations, showing that either all precolorings of the rings of $G$ extend,
or that the rings of $G$ have length $4$ and all their precolorings which extend to a $3$-coloring of the reduced Thomas-Walls graph $T'_4$ also extend to $G$.
This allows us to prove Theorem~\ref{thm-crtri} in Section~\ref{sec-crtri}.  In Section~\ref{sec-qua}, we continue by examining chains that contain a long subchain consisting only
of quadrangulations and show that either every precoloring of the rings extends or the chain is a near $3,3$-quadrangulation.
We combine these results and prove Theorem~\ref{thm-many} in Section~\ref{sec-many}.

\section{Colorings of the special graphs}\label{sec-htw}

In this section we study which precolorings of rings extend in the special graphs appearing in Theorems~\ref{thm-many} and \ref{thm-crtri}.
We examine patched Thomas-Walls graphs in Lemma~\ref{lemma-col-tw}, patched Havel-Thomas-Walls graphs in Lemma~\ref{lemma-col-htw} and tame almost $3,3$-quadrangulations in Lemma~\ref{lemma-col-33}.

We need the following result of Aksenov on the extendability of the precoloring of a $5$-cycle.

\begin{theorem}[Aksenov~\cite{aksenov}]\label{thm-53}
Let $G$ be a graph embedded in the disk with a ring $C=v_1v_2v_3v_4v_5$ tracing its boundary.
Suppose that $G$ contains exactly one triangle $T$.
If $G$ is critical, then all faces of $G$ other than the one bounded by $T$ have length $4$.
Furthermore, if $\psi$ is the $3$-coloring of $C$ given by $\psi(v_1)=\psi(v_3)=1$, $\psi(v_2)=\psi(v_4)=2$
and $\psi(v_5)=3$ and $\psi$ does not extend to a $3$-coloring of $G$, then $v_2$ and $v_3$ are incident with $T$.
\end{theorem}

We also need the following result of Gimbel and Thomassen~\cite{gimbel} on the extendability of the precoloring of a $(\le\!6)$-cycle.

\begin{theorem}[Gimbel and Thomassen~\cite{gimbel}]\label{thm-6cyc}
Let $G$ be a triangle-free graph drawn in the disk with a ring $C$ of length at most $6$ tracing its boundary.
If $G$ is critical, then $|C|=6$ and every face of $G$ has length $4$.

Furthermore, suppose $G'$ is a graph drawn in the disk with a ring $C'=v_1\ldots v_6$ of length $6$ tracing its boundary,
such that all faces of $G'$ have length $4$.  Then a precoloring $\psi$ of $C'$
does not extend to a $3$-coloring of $G'$ if and only if one of the following conditions holds.
\begin{itemize}
\item There exists $i\in\{1,2,3\}$ such that $v_iv_{i+3}\in E(G')$ and $\psi(v_i)=\psi(v_{i+3})$, or
\item $\psi(v_1)=\psi(v_4)\neq\psi(v_2)=\psi(v_5)\neq\psi(v_3)=\psi(v_6)\neq \psi(v_1)$.
\end{itemize}
\end{theorem}

Theorem~\ref{thm-6cyc} has an important consequence for patches.
\begin{corollary}\label{cor-patch}
Let $G$ be a patch with ring $xaybzc$, and let $H$ be the graph with vertex set $\{x,y,z,v\}$ and edges $xv$, $yv$ and $zv$.  Then a $3$-coloring of $\{x,y,z\}$
extends to a $3$-coloring of $G$ if and only if it extends to a $3$-coloring of $H$.
\end{corollary}
\begin{proof}
Let $\psi$ be a $3$-coloring of $\{x,y,z\}$ and let $C=xaybzc$ be the ring of $G$.  If $\psi$ assigns three different colors
to the vertices $x$, $y$, and $z$, so that $\psi$ does not extend to a $3$-coloring of $H$, then
$\psi$ extends uniquely to a $3$-coloring $\psi'$ of $C$, and $\psi'$ does not extend to a $3$-coloring of the patch $G$ by
the second part of Theorem~\ref{thm-6cyc}.  Hence, $\psi$ does not extend to a $3$-coloring of $G$.

Suppose now that $\psi$ extends to a $3$-coloring of $H$, and thus by symmetry we can assume that $\psi(x)=\psi(y)$.
There exists a $3$-coloring $\psi'$ of $C$ extending $\psi$ such that $\psi'(a)\neq \psi'(z)$.  Since $x$ and $y$ have the
same color and one of them is incident to each of $b$ and $c$, we have $\psi'(x)\neq\psi'(b)$ and $\psi'(y)\neq\psi'(c)$. Again by
the second part of Theorem~\ref{thm-6cyc}, $\psi'$ (and thus also $\psi$) extends to a $3$-coloring of $G$.
\end{proof}

That is, replacing a vertex of degree three by a patch does not affect the $3$-colorability of the graph.
We also often use the following mild strenthening of Theorem~\ref{thm-6cyc}.
\begin{lemma}\label{lemma-split}
Let $G$ be a graph drawn in a surface $\Sigma$. Let $K$ be a closed walk of length at most $6$
in $G$ forming the boundary of an open disk $\Lambda\subset\Sigma$, such that
no contractible triangle of $G$ is contained in the closure of $\Lambda$.
Let $G'$ be the subgraph of $G$ drawn in the closure of $\Lambda$.
If a $3$-coloring $\psi$ of $K$ does not extend to a $3$-coloring of $G'$,
then $|K|=6$, $G'$ has a subgraph containing $K$ whose faces in $\Lambda$ all have
length 4,
$K=v_1\ldots v_6$, and either there exists $i\in\{1,2,3\}$ such that $v_iv_{i+3}\in E(G')$ and $\psi(v_i)=\psi(v_{i+3})$, or
$\psi(v_1)=\psi(v_4)\neq\psi(v_2)=\psi(v_5)\neq\psi(v_3)=\psi(v_6)\neq \psi(v_1)$.
\end{lemma}
\begin{proof}
Let $\Delta$ be a closed disk and let $\theta_0$ be a homeomorphism from the interior of $\Delta$ to $\Lambda$ that extends
to a continuous function $\theta$ from $\Delta$ to the closure of $\Lambda$.  Let
$G_\Lambda=\theta^{-1}(G')$ and $K'=\theta^{-1}(K)$.  Note that $G_\Lambda$ is embeded in $\Delta$ and $K'$ is its ring of length $|K|$
tracing the boundary of $\Delta$.  Furthermore, if $K$ is a cycle, then $G_\Lambda$ is isomorphic
to $G'$, and if $K$ is not a cycle ($|K|=6$ and $K$ is a union of two intersecting triangles),
then $G_\Lambda$ is obtained from $G'$ by splitting the vertices appearing multiple times in $K$
in the natural way.

Observe that $\psi'=\psi\circ\theta$
is a $3$-coloring of $K'$ that extends to a $3$-coloring of $G_\Lambda$ if and only if $\psi$ extends to a $3$-coloring of $G$.
Let $G'_\Lambda$ be a critical skeleton of $G_\Lambda$.  If $G'_\Lambda=K'$, then $\psi'$ extends to a $3$-coloring of $G$.
Otherwise, Theorem~\ref{thm-6cyc} implies that $|K'|=6$ and all faces of $G'_\Lambda$ have length $4$, and thus $\theta(G'_\Lambda)$
is a subgraph of $G'$ whose faces in $\Lambda$ all have length $4$.  Furthermore, if $\psi'$ does not extend to a $3$-coloring
of $G'_\Lambda$ (and equivalently, $\psi$ does not extend to a $3$-coloring of $G'$), then $\psi$ must satisfy one of the conditions from the statement
of Lemma~\ref{lemma-split} by the second part of Theorem~\ref{thm-6cyc}.
\end{proof}

Let us give an observation about critical graphs that is often useful.

\begin{lemma}\label{lemma-fr}
Let $G$ be a critical graph drawn in the sphere with holes.
\begin{itemize}
\item Every vertex $v\in V(G)$ that does not belong to the rings has degree at least three.
\item If $K$ is a $(\le\!5)$-cycle in $G$ forming the boundary of an open disk $\Lambda$, and the closure of $\Lambda$
does not contain any contractible triangle of $G$, then $\Lambda$ is a face of $G$.
\item If $K$ is a closed walk of length $6$ in $G$ forming the boundary of an open disk $\Lambda$, and the closure of $\Lambda$
does not contain any contractible triangle of $G$, then either $\Lambda$ is a face of $G$,
or all faces of $G$ contained in $\Lambda$ have length $4$.
\end{itemize}
\end{lemma}
\begin{proof}
Let $C$ be the union of the rings of $G$, and consider any vertex $v\in V(G)\setminus V(C)$.  Suppose for a contradiction
that $v$ has degree at most $2$.  Let $\psi$ be any $3$-coloring of $C$ that extends to a $3$-coloring $\varphi$ of $G-v$.
Then $\psi$ also extends to a $3$-coloring of $G$, by giving the vertices of $V(G)\setminus\{v\}$ the same color as in the
coloring $\varphi$ and by choosing a color of $v$ distinct from the colors of its neighbors.  This contradicts the assumption
that $G$ is critical.

The second and third claim follow similarly using Lemma~\ref{lemma-split}.
\end{proof}

Furthermore, colorings of tents can be described using Theorem~\ref{thm-6cyc}.
Consider a $4$-cycle $C=u_1u_2u_3u_4$ and its $3$-coloring $\psi$.  We say that $\psi$ is \emph{$u_1$-diagonal} if $\psi(u_1)\neq \psi(u_3)$,
and that it is \emph{bichromatic} if $\psi(u_1)=\psi(u_3)$ and $\psi(u_2)=\psi(u_4)$.  Note that every $3$-coloring of $C$ is $u_1$-diagonal, $u_2$-diagonal, or bichromatic.
The following claim is proved analogously to Corollary~\ref{cor-patch}.
\begin{corollary}\label{cor-col-pt}
If $G$ is a tent with the ring $C=v_1v_2v_3v_4$, then exactly the $v_1$-diagonal and $v_2$-diagonal colorings of $C$
extend to $3$-colorings of $G$.
\end{corollary}

Let $G_0$ be either a reduced Thomas-Walls graph, or a Havel-Thomas-Walls graph.
Let $S$ be an independent set of vertices of $G_0$ of degree three.
Let $C_0=v_1v_2v_3v_4$ be a ring of $G_0$, with the interface pair $v_1v_3$.  Let $G_1$ be obtained from $G_0$ by replacing the vertices of $S$ by patches,
and let $G$ be obtained from $G_1$ by framing on its interface pairs.  Let $C=v_1v'_2v_3v'_4$ be the ring of $G$ corresponding to $C_0$.
We say that the ring $C$ is \emph{strong} if $G\neq C$, $v_2,v_4\not\in S$, $v'_2=v_2$, and $v'_4=v_4$, that is, $v_2$ and $v_4$ are not
affected by patching or created by framing.  Otherwise, we say that $C$ is \emph{weak}.
A $3$-coloring $\psi$ of $C$ is \emph{dangerous} if either $\psi$ is $v_1$-diagonal, or $C$ is strong and $\psi$ is bichromatic.

Let us first deal with Thomas-Walls graphs.

\begin{lemma}\label{lemma-col-tw}
Let $n\ge 1$ be an integer, let $H$ be a patched Thomas-Walls graph obtained from $T'_n$ by patching, and
let $G$ be a graph obtained by framing on interface pairs $u_1u_3$ and $v_1v_3$ of $H$.
Let $C_1=u_1u_2u_3u_4$ and $C_2=v_1v_2v_3v_4$ be the rings of $G$ and let $\psi$ be a precoloring of $C_1\cup C_2$.
If $\psi$ extends to a $3$-coloring of $G$, then it is dangerous on at most one of $C_1$ and $C_2$.
Furthermore, if $n\ge 4$ and $\psi$ is not dangerous on both $C_1$ and $C_2$, then $\psi$ extends to a $3$-coloring of $G$.
\end{lemma}
\begin{proof}
Firstly, suppose that $G$ is the reduced Thomas-Walls graph $T'_n$.  We proceed by induction on $n$.  
The claims obviously hold when $n=1$.  Hence, assume that $n\ge 2$, and thus both $C_1$ and $C_2$ are strong.
Let $G$ be obtained from a copy $G'$ of $T'_{n-1}$ with rings $C_1$ and $C'_2=v'_1v'_2v'_3v_4$ (with interface pairs $u_1u_3$ and $v'_2v_4$)
by adding the ring $C_2$ and the edge $v_2v'_2$.  Let $\psi$ be a $3$-coloring of $C_1\cup C_2$.

Suppose for a contradiction that $\psi$ is dangerous on both $C_1$ and $C_2$ and extends to a $3$-coloring $\varphi$ of $G$; then $\psi(v_2)=\psi(v_4)$, and because of the edge $v_2v'_2$, $\varphi$ is $v'_2$-diagonal
on $C'_2$, and thus it is dangerous on $C'_2$.  This is a contradiction by the induction hypothesis for $G'$.
Therefore, if $\psi$ extends to a $3$-coloring of $G$, then it is dangerous on at most one of $C_1$ and $C_2$.

Suppose now that $\psi$ is dangerous on at most one of $C_1$ and $C_2$, and $n\ge 4$.
We need to show that $\psi$ extends to a $3$-coloring of $G$.
By a straightforward case analysis, this is true when $n=4$, and thus assume that $n\ge 5$.
By symmetry, we can assume that $\psi$ is not dangerous on $C_2$, and thus $\psi(v_2)\neq \psi(v_4)$.  Color
$v'_2$ by $\psi(v_4)$ and give $v'_1$ and $v'_3$ distinct colors; the obtained coloring of $C'_2$ is $v'_1$-diagonal, and thus it is not dangerous on $C'_2$.  By the induction hypothesis, we can
extend the coloring to $G'$.  This gives a $3$-coloring of $G$ extending $\psi$.

Therefore, the claim holds for reduced Thomas-Walls graphs.  Suppose that $G$ is obtained from $G_0=T'_n$ by patching on an independent set $S$
and framing on the interface pairs.
Let $S_1$ be the subset of $S$ consisting of vertices not incident with the rings.
Let $G_1$ be the graph such that $G_1$ is obtained from $G_0$ by patching on $S_1$ and $G$ is obtained from $G_1$ by patching on $S\setminus S_1$ and
framing on the interface pairs.  By Corollary~\ref{cor-patch} and the previous analysis of $T'_n$, the graph $G_1$ satisfies the conclusions of
Lemma~\ref{lemma-col-tw}.

Since the vertices of the interface pairs of $G_0$ have degree two, they do not belong to $S$.
We define a coloring $\psi'$ of the rings of $G_1$ as follows.
Suppose that $C'_1=u_1u'_2u_3u'_4$ is a ring of $G_1$, where $u'_2$ has degree three.
Let $\psi'(u_1)=\psi(u_1)$ and $\psi'(u_3)=\psi(u_3)$.
If $\psi$ is $u_1$-diagonal on $C_1$, or if $u'_2\not\in S$, $u_2=u'_2$, and $u_4=u'_4$, 
then let $\psi'(u'_2)=\psi(u_2)$ and $\psi'(u'_4)=\psi(u_4)$.  Otherwise, choose $\psi'(u'_2)$ and $\psi'(u'_4)$
so that $\psi'_i$ is $u'_2$-diagonal on $C'_1$ and so that $\psi'(u'_i)=\psi(u_i)$ for all $i\in\{2,4\}$ such that $u_i=u'_i$.
Define $\psi'$ on the other ring $C'_2$ of $G_1$ analogously.
Note that for $i\in\{1,2\}$, $C'_i$ is strong in $G_1$, and if $C'_i\neq C_i$, then $C_i$ is weak in $G$; hence, $\psi$ is dangerous on $C_i$ if and only if $\psi'$
is dangerous on $C'_i$.

Suppose first for a contradiction that $\psi$ is dangerous both on $C_1$ and $C_2$, and that it extends to a $3$-coloring $\varphi$ of $G$.
\begin{itemize}
\item If $C_1$ is weak, then $\psi$ is $u_1$-diagonal on $C_1$, and $\psi'(u'_i)=\varphi(u'_i)=\psi(u_i)$ for $i\in\{2,4\}$. Let $z$ be the neighbor of $u'_2$ distinct from $u_1$ and $u_3$.
If $u'_2\not\in S$, then $u'_2z$ is an edge of $G$, and $\varphi(z)\neq\psi'(u'_2)$.  If $u'_2\in S$, then $\varphi(z)\neq\psi'(u'_2)$ by the second part of Theorem~\ref{thm-6cyc} applied to the patch $P$
replacing $u'_2$ and the $3$-coloring of $P$ given by $\varphi.$
\item If $C_1$ is strong, then $u'_2=u_2$, $u'_4=u_4$, and $u'_2z$ is an edge of $G$.
\end{itemize}
Together with a similar argument applied at $C_2$, we conclude that $\psi'\cup(\varphi\restriction V(G_1))$
is a $3$-coloring of $G_1$ extending $\psi'$, which is a contradiction since $\psi'$ is dangerous both on $C'_1$ and $C'_2$
and the conclusions of Lemma~\ref{lemma-col-tw} are satisfied by $G_1$ as we argued before.

Next, suppose that $\psi$ is not dangerous on one of the rings, say on $C_1$, and that $n\ge 4$. Then $\psi'$ is not dangerous on $C'_1$, and thus it extends to a $3$-coloring
$\varphi'$ of $G_1$.  If $u'_2\not\in S$, then let $\varphi_1$ be an empty coloring.
If $u'_2\in S$, then by Theorem~\ref{thm-6cyc}, there exists a $3$-coloring $\varphi_1$ of
the patch replacing $u'_2$ such that $\varphi_1(x)=\psi(x)$ for $x\in\{u_1,u_2,u_3\}$ and $\varphi_1(z)=\varphi'(z)$, where $z$ is the neighbor of $u'_2$ distinct from $u_1$ and $u_3$.
Let $\varphi_2$ be defined analogously at $C_2$.  Observe that $\psi\cup\varphi_1\cup \varphi_2\cup (\varphi'\restriction V(G))$ is a $3$-coloring of $G$ extending $\psi$.
\end{proof}

Using this lemma, we can easily handle Havel-Thomas-Walls graphs as well.

\begin{lemma}\label{lemma-col-htw}
Let $G$ be a graph obtained by framing on the interface pair $v_1v_3$ of a patched Havel-Thomas-Walls graph,
with ring $C=v_1v_2v_3v_4$.  A $3$-coloring $\psi$ of $C$ extends to a $3$-coloring of $G$ if and only if $\psi$ is not dangerous on $C$.
\end{lemma}
\begin{proof}
Consider first the case that $G$ is a Havel-Thomas-Walls graph, obtained from the reduced Thomas-Walls graph $T'_n$ with rings $C$
and $C'=u_1u_2u_3u_4$ and interface pairs $v_1v_3$ and $u_1u_3$ by either adding the edge $u_1u_3$, or the graph in Figure~\ref{fig-havel}(a).
Consider any $3$-coloring $\varphi$ of $G$.  Note that both the edge $u_1u_3$ and the graph from Figure~\ref{fig-havel}(a) ensure that
$\varphi(u_1)\neq\varphi(u_3)$, and thus $\varphi$ is $u_1$-diagonal on $C'$.  Consequently, $\varphi$ is dangerous on $C'$,
and by Lemma~\ref{lemma-col-tw}, it is not dangerous on $C$.  Therefore, if $\psi$ extends to a $3$-coloring of $G$, then $\psi$ is not dangerous on $C$.

Conversely, if $\psi$ is not dangerous on $C$, then we can extend it to a $3$-coloring of $T'_n$ that is $u_1$-diagonal on $C'$
(for $n\ge 4$, this follows by Lemma~\ref{lemma-col-tw}; for $1\le n\le 3$, it is easy to construct the desired colorings),
and further extend the coloring to the graph from Figure~\ref{fig-havel}(a) if present in $G$.

The case that $G$ is obtained from a Havel-Thomas-Walls graph by patching and framing is handled in the same way as in the proof of Lemma~\ref{lemma-col-tw}.
\end{proof}

Next, we consider the colorability of near $3,3$-quadrangulations.  We need some additional definitions.
Given a $3$-coloring $\psi:V(G)\to\{0,1,2\}$ of a graph $G$, let us define an orientation $G_\psi$ of $G$ by orienting every edge $uv\in E(G)$ towards
$v$ if and only if $\psi(v)-\psi(u)\in\{1,-2\}$.
Let $W=v_1v_2\ldots v_k$ be a walk in a graph $G$.  We define $\omega(W,\psi)$ to be the difference between the number of forward and backward edges
of $W$ in $G_\psi$, i.e.,
$$\omega(W,\psi)=|\{i:1\le i\le k-1,v_iv_{i+1}\in E(G_\psi)\}|-|\{i:1\le i\le k-1,v_{i+1}v_i\in E(G_\psi)\}|.$$
Suppose that $G$ is embedded in the cylinder, and let $C$ be a non-contractible cycle in $G$.  We fix one orientation around the cylinder as \emph{positive}.
Let $W$ (with $v_1=v_k$) be a closed walk tracing $C$ in the positive direction.  We define the \emph{winding number} of $\psi$ on $C$ as $\omega(W,\psi)/3$.
Observe that the winding number is an integer of the same parity as the length of $C$.
We need the following result concerning the colorability of quadrangulations, see e.g. the Propositions 4.1 and 4.2 in~\cite{trfree5}.

\begin{lemma}\label{lemma-cylflow}
Let $G$ be a graph embedded in the cylinder with rings $C_1$ and $C_2$, such that all non-ring faces of $G$ have length $4$.
For any $3$-coloring $\varphi$ of $G$, the winding number of $\varphi$ on $C_1$ is equal to the winding number of $\varphi$ on $C_2$.
\end{lemma}

We also need a strengthening of Lemma~\ref{lemma-cylflow} for quadrangulations of the cylinder with rings of length at most $4$.

\begin{lemma}\label{lemma-col44}
Let $G$ be a tame graph embedded in the cylinder with rings $C_1$ and $C_2$ of length at most $4$.  If all faces of $G$ have length $4$,
the distance between $C_1$ and $C_2$ is at least $|C_1|$, and $\psi$ is a precoloring of $C_1\cup C_2$ that does not extend to a $3$-coloring of $G$,
then $|C_1|=|C_2|=3$ and $\psi$ has opposite winding numbers on $C_1$ and $C_2$, i.e. $\psi$ has winding number $+1$ on one of them
and $-1$ on the other one.
\end{lemma}
\begin{proof}
Note that $C_1$ and $C_2$ have the same parity, and thus $|C_1|=|C_2|$.  Suppose first that $|C_1|=4$.  In this case, $G$ is bipartite.
Let $C_1=v_1v_2v_3v_4$, with the labels chosen so that $\psi(v_1)=\psi(v_3)$.  For $0\le i\le 3$, let $S_i$ denote the set of vertices
of $G$ at distance exactly $i$ from $\{v_2,v_4\}$.  Let $G'$ be the graph obtained from $G-(S_0\cup S_1\cup S_2)$
by identifying all vertices in $S_3$ to a single vertex $x$.  Note that $G'$ is also bipartite, and in particular it has no loops
and it is triangle-free.  Furthermore, since the distance between $C_1$ and $C_2$ is at least four, $C_2$ is a cycle in $G'$.  By Lemma~\ref{lemma-split},
the graph $G'$ has a $3$-coloring $\varphi'$ that matches $\psi$ on $C_2$.  Then, we can obtain a $3$-coloring of $G$ that extends $\psi$
by coloring every vertex $v\in V(G')\setminus\{x\}$ by the color $\varphi'(v)$, all vertices in $S_3$ by
the color $\varphi'(x)$, all vertices in $S_2$ by a color distinct from $\varphi'(x)$ and $\psi(v_1)$,
all vertices in $S_1$ by the color $\psi(v_1)$, and $v_2$ by $\psi(v_2)$ and $v_4$ by $\psi(v_4)$.

Suppose now that $|C_1|=3$.  Note that the winding number of $\psi$ on each of $C_1$ and $C_2$ is either $+1$ or $-1$.
By Lemma~\ref{lemma-cylflow}, if the winding numbers of $\psi$ on $C_1$ and $C_2$ are opposite, then $\psi$ does not extend to a $3$-coloring of $G$.
Hence, we can assume that the winding numbers of $\psi$ on $C_1$ and $C_2$ are the same. 
For $i\in\{1,2\}$, let $C_i=v_{i,1}v_{i,2}v_{i,3}$, with the labels chosen so that $\psi(v_{i,j})=j$ for $j\in\{1,2,3\}$.
For $j\in\{1,2,3\}$, let $S_{i,j}$ denote the set of vertices of $G$ adjacent to $v_{i,j}$
that do not belong to $V(C_i)$.  Since $G$ is tame and the distance between $C_1$ and $C_2$ is at least three,
these sets are pairwise disjoint.  Let $G'$ be the graph obtained from $G$ by, for $i\in\{1,2\}$ and $j\in\{1,2,3\}$, identifying
all vertices in $S_{i,j}$ to a single vertex $x_{i,j}$, and suppressing the resulting faces of length two.
Note that $K_i=x_{i,1}x_{i,2}x_{i,3}$ is a non-contractible triangle in $G'$.  If $\psi$ extends to a $3$-coloring of $G'$,
then it also extends to a $3$-coloring of $G$, obtained by giving each vertex in $S_{i,j}$ the color of $x_{i,j}$.

Let $G_0$ be the subgraph of $G'$ drawn between $K_1$ and $K_2$.  By Theorem~\ref{thm-onetri},
there exists a $3$-coloring $\varphi_1$ of $G_0$ such that $\varphi_1(x_{1,j})=(j\bmod 3)+1$ for $j\in\{1,2,3\}$.  By permuting the colors in $\varphi_1$,
we obtain a $3$-coloring $\varphi_2$ of $G_0$ such that $\varphi_2(x_{1,j})=((j+1)\bmod 3)+1$ for $j\in\{1,2,3\}$ and $\varphi_1(v)\neq\varphi_2(v)$ for every $v\in V(G_0)$.
Suppose that $\varphi_1\cup \psi$ is not a $3$-coloring of $G'$.  By Lemma~\ref{lemma-cylflow}, the winding numbers of $\varphi_1$ on
$K_1$ and $K_2$ are the same, and thus $\varphi_1(x_{2,j})=\psi(v_{2,j})$ for $j\in\{1,2,3\}$.  Then, it follows that $\psi\cup\varphi_2$ is a $3$-coloring of $G'$.
We conclude that $G'$, and thus also $G$, has a $3$-coloring which extends $\psi$.
\end{proof}

Let $G$ be a near $3,3$-quadrangulation, let $C$ be one of the rings
of $G$ and let $\psi$ be a $3$-coloring of $C$.  Let us discuss several cases:
\begin{itemize}
\item If $C$ shares an edge with a triangle $T$ in $G$ (where possibly $T=C$ if $C$ is a triangle),
then $\psi$ uniquely extends to a $3$-coloring $\varphi$ of $T$.  Let $w$ be the winding number of $\varphi$ on $T$.
In this case, we say that $\psi$ on $C$ \emph{causes winding number $w$}.
\item If $C$ does not share an edge with a triangle and we can label the vertices of $C$ as $v_1v_2v_3v_4$ so that the path $v_1v_2v_3$ is a
part of the boundary of a $5$-face $f$ and $\psi(v_1)\neq \psi(v_3)$, then draw an edge between $v_1$ and $v_3$
in $f$, and let $w$ be the winding number of $\psi$ on the triangle $v_1v_3v_4$.  In this case, we also say that $\psi$ on $C$ \emph{causes winding number $w$}.
\item Otherwise, we say that $\psi$ on $C$ \emph{does not cause fixed winding number}.
\end{itemize}
If $C_1$ and $C_2$ are the rings of $G$ and $\psi$ is their precoloring, we say that $\psi$ is \emph{inconsistent} if
$\psi$ causes winding numbers on both $C_1$ and $C_2$ and these winding numbers are opposite.
Otherwise, $\psi$ is \emph{consistent}.

\begin{lemma}\label{lemma-col-33}
Let $G$ be a tame near $3,3$-quadrangulation embedded in the cylinder with rings $C_1$ and $C_2$.
If a precoloring $\psi$ of $C_1\cup C_2$ extends to a $3$-coloring of $G$, then it is consistent.
Furthermore, if the distance between $C_1$ and $C_2$ is at least $9$, then every consistent precoloring
of $C_1\cup C_2$ extends to a $3$-coloring of $G$.
\end{lemma}
\begin{proof}
Let $G_0$ be a graph and $\psi_0$ a $3$-coloring of its rings obtained from $G$ and $\psi$ as follows.
For $i=1,2$, if $C_i$ shares an edge with a triangle $T$,
then remove all vertices between $C_i$ and $T$ (excluding $T$, but including $V(C_i)\setminus V(T)$),
and let $\psi_0$ restricted to $T$ be the unique $3$-coloring of $T$ that matches $\psi$ on $V(T)\cap V(C_i)$.
If $C_i=v_1v_2v_3v_4$, $v_1v_2v_3$ is a part of the boundary
of a $5$-face $f$, and $\psi(v_1)\neq \psi(v_3)$, then remove $v_2$ and add the edge $v_1v_3$ drawn inside $f$,
and let $\psi_0$ restricted to $v_1v_3v_4$ match $\psi$.  Finally, if $\psi(v_1)=\psi(v_3)$, then do not alter $G$ at $C_i$
and let $\psi_0$ restricted to $C_i$ match $\psi$.
Let $C'_1$ and $C'_2$ be the rings of $G_0$ obtained from $C_1$ and $C_2$, respectively.

Using Theorem~\ref{thm-onetri}, observe that $\psi$ extends to a $3$-coloring of $G$ if and only if $\psi_0$ extends to a $3$-coloring
of $G_0$.  Note that if $\psi$ is inconsistent, then $G_0$ is a $3,3$-quadrangulation and the winding numbers of $\psi_0$
on $C'_1$ and $C'_2$ are opposite, and thus by Lemma~\ref{lemma-cylflow}, $\psi_0$ does not extend to a $3$-coloring of $G_0$.

Therefore, it suffices to consider the case that $\psi$ is consistent and the distance between $C_1$ and $C_2$ in $G$ is at least $9$,
and to show that in this case, $\psi_0$ extends to a $3$-coloring of $G_0$.
If $\psi$ on $C_1$ or $C_2$ causes winding number, then let $w$ be this winding number.  Otherwise, let $w=1$.

We now modify $G_0$ and $\psi_0$ into another auxiliary graph $G_1$ with a precoloring $\psi_1$ of its rings.
For each $i=1,2$ such that $|C'_i|=4$, we modify the graph as follows.  Let $C'_i=v_1v_2v_3v_4$, where $v_1v_2v_3u_4u_5$ is
a $5$-face and $\psi_0(v_1)=\psi_0(v_3)$.  Let $\varphi$ be the unique $3$-coloring of the $5$-cycle $K=v_1u_5u_4v_3v_4$ matching $\psi_0$
on $v_1$, $v_3$, and $v_4$, such that $\varphi$ has winding number $w$ on $K$.  If necessary, exchange the labels of $v_1$ with $v_3$ and of $u_4$ with $u_5$ so that $\varphi(v_4)=\varphi(u_5)$.
If $v_1$ is contained in a triangle $T$, then remove from $G_1$ all vertices between $C'_i$ and $T$ (excluding $T$, but including $V(C'_i)\setminus V(T)$),
and let $\psi_1$ restricted to $T$ be the $3$-coloring with winding number $w$ such that $\psi_0(v_1)=\psi_1(v_1)$
(since $T$ does not share an edge with $C'_i$, $\varphi(v_1)=\varphi(v_3)$, and $\varphi(v_4)=\varphi(u_5)$,
the coloring $\varphi$ extends to a $3$-coloring of the subgraph of $G_0$ drawn between
$K$ and $T$ by Theorem~\ref{thm-53}, and the restriction of this $3$-coloring to $T$ matches $\psi_1$ by Lemma~\ref{lemma-cylflow}).
If $v_1$ is not contained in any triangle, then remove $v_1$ and $v_2$ and identify all remaining neighbors of $v_1$ to a single vertex
$x$, and let $\psi_1$ restricted to $xu_4v_3$ be the $3$-coloring such that $\psi_1(v_3)=\varphi(v_3)$, $\psi_1(u_4)=\varphi(u_4)$
and $\psi_1(x)=\varphi(v_4)=\varphi(u_5)$.
For each $i=1,2$ such that $|C'_i|=3$, we do not modify the graph and we let $\psi_1$ restricted to $C'_i$ match $\psi_0$.

Observe that $G_1$ is a $3,3$-quadrangulation and that if $\psi_1$ extends to a $3$-coloring of $G_1$, then we can obtain a $3$-coloring
of $G_0$ that extends $\psi_0$.  Let $C''_1$ and $C''_2$ be the rings of $G_1$.
Let $T_1$ and $T_2$ be non-contractible triangles in $G_1$
such that the subgraph $G'_1$ drawn between $T_1$ and $T_2$ is tame and $G'_1$ is as large as possible, with labels chosen so that $T_1$ separates
$C''_1$ from $T_2$.  Since $G$ is tame and the distance between $C_1$ and $C_2$ in $G$ is at least $9$, the construction
of $G_0$ and $G_1$ ensures that the distance between $T_1$ and $T_2$ in $G'_1$ is at least three.

By Theorem~\ref{thm-onetri}, $\psi_1$ extends to a $3$-coloring $\varphi_1$ of the subgraphs of $G_1$ drawn between $C''_1$ and $T_1$,
and between $C''_1$ and $T_2$, and by Lemma~\ref{lemma-cylflow}, $\varphi_1$ has winding number $w$ both on $T_1$ and $T_2$.
By Lemma~\ref{lemma-col44}, the restriction of $\varphi_1$ to $T_1\cup T_2$ extends to a $3$-coloring $\varphi_2$ of $G'_1$.
Hence, $\psi_1$ extends to a $3$-coloring $\varphi_1\cup\varphi_2$ of $G_1$.  This also gives a $3$-coloring of $G$ extending $\psi$.
\end{proof}

\section{Basic cylinders}\label{sec-bc}

We now make the first step towards the proof of Theorem~\ref{thm-many}, by studying the coloring properties of chains of graphs.
Let $G$ and $G'$ be graphs embedded in a surface with the same rings, and let $C$ be the union of the rings.
We say that $G$ \emph{dominates} $G'$ if every precoloring of $C$ that extends to a $3$-coloring of $G$ also extends to a $3$-coloring of $G'$.
We aim to show that for any chain $G$ with sufficiently many graphs that are not quadrangulations, either all precolorings of the rings of $G$ extend,
or $G$ is dominated by the reduced Thomas-Walls graph $T'_4$.  To do so, we contract $4$-faces in the graphs of the chain as long as possible,
ending up with easily enumerated and analyzed ``basic'' graphs.

Let us give a few more definitions.
We say that a graph $H$ embedded in the cylinder with rings $K_1$ and $K_2$ is a \emph{quadrangulation} if all its non-ring faces
have length $4$.
We say that $H$ is \emph{quadrangulated} if it contains a quadrangulation with rings $K_1$ and $K_2$ as a subgraph.

Let $H$ be a graph embedded in the cylinder without contractible triangles, and let $T_y=xy_1y_2$ and $T_z=xz_1z_2$ be
triangles in $H$ with the vertices listed in the positive direction around the cylinder.
By \emph{collapsing} $T_y$ and $T_z$, we mean removing the vertices and edges of $H$ that are separated from the rings by $T_y\cup T_z$,
and identifying $y_i$ with $z_i$ for $i=1,2$.  Let $H'$ be obtained from $H$ by collapsing $T_y$ and $T_z$.
By Lemma~\ref{lemma-split} (applied to the subgraph of $H$ drawn between $T_y$ and $T_z$),
every $3$-coloring of $H'$ extends to a $3$-coloring of $H$, and thus
$H'$ dominates $H$.

Let $G$ be a tame chain of graphs $G_1,\ldots,G_n$ with the sequence of cutting cycles $\mathcal{C}=C_0,\ldots,C_n$.
We say that the vertices of the cutting cycles are \emph{special}.
Let $f=x_1x_2x_3x_4$ be a $4$-face in $G_i$ for some $i\in\{1,\ldots, n\}$.
We call the identification of $x_1$ and $x_3$ to a new vertex $x$ \emph{legal} if
\begin{itemize}
\item at most one of $x_1$ and $x_3$ is special, and
\item $G_i$ does not contain (not necessarily distinct) paths $x_1z_1z_2x_3$ and $x_1z'_1z'_2x_3$ such that
$\{x_1,z_1,z_2,x_3\}\cap V(C_{i-1})\neq\emptyset$ and $\{x_1,z'_1,z'_2,x_3\}\cap V(C_i)\neq\emptyset$.
\end{itemize}
If a graph $H$ is obtained from a graph $G$ by a legal identification, then $H$ dominates $G$ since every 3-coloring of $H$ extends to
a $3$-coloring of $G$.
Note that if the indentification of $x_1$ with $x_3$ creates parallel edges, we suppress them, i.e., keep only
one (arbitrary) of them.  Since $G$ is tame, $x_1$ and $x_3$ are not adjacent, and thus the identification does
not create loops.  Also, by the second condition of the legality, we can collapse the triangles possibly created by the identification
while keeping the rings of $G_i$ disjoint.

Let $G$ be a tame graph embedded in the cylinder with rings of length at most $4$.
We say that $G$ is \emph{basic} if either $G$ contains no contractible $4$-cycle, or it is one
of the graphs depicted in Figure~\ref{fig-basic} (the rings of each depicted basic graph are the
cycle bounding its outer face and the cycle bounding its central face of length $3$ or $4$).
We aim to show that it suffices to consider the chains of basic graphs.

\begin{figure}
\begin{center}
\includegraphics[scale=0.9,page=1]{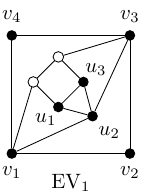}
\includegraphics[scale=0.9,page=2]{fig-basic}
\includegraphics[scale=0.9,page=3]{fig-basic}
\includegraphics[scale=0.9,page=4]{fig-basic}
\includegraphics[scale=0.9,page=5]{fig-basic}
\includegraphics[scale=0.9,page=6]{fig-basic}
\includegraphics[scale=0.9,page=7]{fig-basic}
\includegraphics[scale=0.9,page=8]{fig-basic}
\includegraphics[scale=0.9,page=9]{fig-basic}
\includegraphics[scale=0.9,page=10]{fig-basic}
\includegraphics[scale=0.9,page=11]{fig-basic}
\includegraphics[scale=0.9,page=12]{fig-basic}
\includegraphics[scale=0.9,page=13]{fig-basic}
\includegraphics[scale=0.9,page=14]{fig-basic}
\includegraphics[scale=0.9,page=15]{fig-basic}
\includegraphics[scale=0.9,page=16]{fig-basic}
\includegraphics[scale=0.9,page=17]{fig-basic}
\includegraphics[scale=0.9,page=18]{fig-basic}
\includegraphics[scale=0.9,page=19]{fig-basic}
\includegraphics[scale=0.9,page=20]{fig-basic}
\includegraphics[scale=0.9,page=21]{fig-basic}
\end{center}
\caption{Basic graphs with contractible 4-cycles. The dotted lines indicate how to identify vertices in quadrangulations.}\label{fig-basic}
\end{figure}

We need the following result concerning extension of a precoloring of a $(\le\!9)$-cycle.
\begin{theorem}[Thomassen~\cite{thom-torus}]\label{thm-girth5}
Let $G$ be a graph of girth at least $5$ drawn in the disk with the ring $C$ of length at most $8$.
If $G$ is critical, then $|C|=8$ and $G=C+e$ for an edge $e$ with both ends in $C$.
\end{theorem}

Borodin et al.~\cite{bor7c} investigated precolorings of $7$-faces in planar triangle-free graphs, and their result has the following corollary.

\begin{theorem}[Borodin et al.~\cite{bor7c}]\label{thm-7cyc}
Let $G$ be a triangle-free graph embedded in the disk with the ring $C$ of length $7$ tracing its boundary.
If $G$ is critical, then it has exactly one $5$-face $f$, all faces other than $f$ have length $4$,
and the cycle bounding $f$ intersects $C$ in a path of length at least two.
Furthermore, if $\psi$ is a precoloring of $C$ that does not extend to a $3$-coloring of $G$
and $xyz$ is a subpath of $C$ such that $\psi(x)=\psi(z)$, then $y$ is incident with $f$.
\end{theorem}

We now turn our attention to legal identifications.
Let $G$ be a chain of critical graphs $G_1,\ldots,G_n$ with the sequence of cutting cycles $\mathcal{C}=C_0,\ldots,C_n$.
Let $r(G,\mathcal{C})$ denote the number of the graphs among $G_1$, \ldots, $G_n$ that are not quadrangulations.

\begin{lemma}\label{lemma-legalidentification}
Let $G$ be a tame graph embedded in the cylinder with rings of length at most $4$.
Suppose that $G$ is chain of critical graphs $G_1,\ldots,G_n$ with cutting cycles $\mathcal{C} =C_0,\ldots,C_n$.
If there exists $i\in \{1,\ldots,n\}$ such that $G_i$ contains a $4$-face allowing a legal identification,
then there exists a tame graph $H$ embedded in the cylinder with the same rings,
such that
\begin{itemize}
\item $|V(H)|<|V(G)|$ and $H$ dominates $G$, and
\item $H$ is a chain of at least $n$ critical graphs with cutting cycles $\mathcal{C}'$ satisfying $r(H,\mathcal{C}')\geq r(G,\mathcal{C})$.
\end{itemize}
\end{lemma}
\begin{proof}
Let $f=x_1x_2x_3x_4$ be a $4$-face allowing a legal identification of $x_1$ and $x_3$ in $G_i$.
Let a chain $H$ with cutting cycles $\mathcal{C'}$ be obtained from $G$ as follows.
First, identify $x_1$ with $x_3$ to a new vertex $x$, obtaining a graph $H'$.
Collapse all intersecting non-contractible triangles in $H'$ to a single triangle, obtaining a graph $H''$.
Note that $H''$ contains at most one triangle $T$ not belonging to $\mathcal{C}$.
Furthermore, by the legality of the identification, the cycles $C_{i-1}$ and $C_i$ in $H''$ are vertex-disjoint and
at most one of them (of length $4$) intersects $T$.
If $T$ does not exist, then let $\mathcal{C}'=\mathcal{C}$.  Otherwise,
let $\mathcal{C'}=C'_0,\ldots,C'_{n'}$ be the sequence of cutting cycles in $H''$ obtained from $\mathcal{C}$
by adding $T$ and removing the $4$-cycle distinct from $C_0$ and $C_n$ that intersects $T$, if any.
For $j=1,\ldots, n'$, let $H''_j$ denote the subgraph of $H''$ drawn between $C'_{j-1}$ and $C'_j$.
Finally, replace $H''_j$ by its maximal critical subgraph $H_j$ for $j=1,\ldots, n'$, obtaining the graph $H$.

Suppose for a contradiction that $H'$ contains a contractible triangle $xz_1z_2$.  Then
$x_1x_2x_3z_1z_2$ is a contractible $5$-cycle in $G_i$, and by Lemma~\ref{lemma-fr}, this $5$-cycle bounds
a face of $G_i$.  Hence, $x_2$ has degree two in $G_i$, and by Lemma~\ref{lemma-fr}, $x_2$ is special.
But then also both neighbors of $x_2$ would be special, which contradicts the assumption that the identification of $x_1$
and $x_3$ is legal.  Therefore, $H'$ contains no contractible triangles, and consequently, $H''$ and $H$ are tame.
Clearly, $H$ dominates $G$ and $|V(H)|<|V(G)|$.
Note that $\mathcal{C}'$ contains all the triangles of $H$, and that $|\mathcal{C}'|\ge |\mathcal{C}|$.
Hence, if $r(H,\mathcal{C}')\ge r(G,\mathcal{C})$, then $H$ satisfies all the conditions of Lemma~\ref{lemma-legalidentification}.
Let us argue that we can choose the face $f$ and its labelling so that the identification of $x_1$ and $x_3$ is legal and
$r(H,\mathcal{C}')\ge r(G,\mathcal{C})$.  We distinguish several cases.

Suppose first that $f=x_1x_2x_3x_4$ is a face of $G_i$ such that at most one of $x_1$ and $x_3$ is special and
$G_i$ contains no path $x_1z_1z_2x_3$ of length three disjoint with $\{x_2,x_4\}$.  Note that this implies that the identification
of $x_1$ with $x_3$ is legal.  Consider the procedure described at the beginning of the proof.  In this case $H''=H'$ and $\mathcal{C}'=\mathcal{C}$,
and $H$ is obtained from $G$ by replacing $G_i$
with $H_i$.  Suppose that $H_i$ is a quadrangulation.  For every $4$-face $h$ of $H_i$, either $h$ corresponds to a $4$-face of $G_i$,
or $h$ corresponds to a contractible $6$-cycle $K$ in $G_i$ containing the path $x_1x_2x_3$.  In the latter case,
note that $K$ does not bound a face, since then $x_2$ would have degree two and Lemma~\ref{lemma-fr} would imply that
$x_1$, $x_2$ and $x_3$ are special vertices, contrary to the assumption that the identification of $x_1$ and $x_3$ is legal.
Hence, Lemma~\ref{lemma-fr} implies that all faces of $G_i$ inside $K$ have length $4$.  We conclude that if $H_i$ is a
quadrangulation, then $G_i$ is a quadrangulation as well, and thus $r(H,\mathcal{C}')\ge r(G,\mathcal{C})$.

Let us remark that if the identification of $x_1$ with $x_3$ does not create a new triangle, then the legality only
requires that not both $x_1$ and $x_3$ are special.  Therefore, it remains to consider the case that the following holds.
\claim{cl-alltriang}{For every $4$-face $y_1y_2y_3y_4$ in $G_i$ such that not both $y_1$ and $y_3$ are special,
there exists a path in $G_i$ of length three joining $y_1$ with $y_3$ and disjoint with $\{y_2,y_4\}$.}

Next, suppose that $f=x_1x_2x_3x_4$ is a face in $G_i$ such that $x_3$ and $x_4$ are not special.
By (\ref{cl-alltriang}), $G$ contains paths of length three between $x_1$ and $x_3$, and between $x_2$ and $x_4$,
and the two paths must intersect, forming a triangle.  Since all triangles are contained in $\mathcal{C}$,
this triangle is equal to say $C_i$ and contains $x_1$ and $x_2$.  Hence, $G$ contains paths $x_1zy_1x_3$ and $x_2zy_2x_4$,
where $C_i=x_1x_2z$.  At least one of the $4$-cycles $zx_2x_3y_1$ and $zx_1x_4y_2$ (say the former) is contractible.
By Lemma~\ref{lemma-fr}, $zx_2x_3y_1$ is a face.  Note that since $G$ is tame and $x_3$ and $x_4$ are not
special, $y_1$ is not adjacent to $y_2$ or $x_4$.  Consequently, $G$ contains no path of length three between $y_1$ and $x_2$
disjoint with $\{z,x_3\}$, and by (\ref{cl-alltriang}), it follows that $y_1\in V(C_{i-1})$.  Therefore, the $4$-cycle
$zx_1x_4y_2$ is also contractible, and by symmetry, $y_2\in V(C_{i-1})$.  However, then $G_i$ is isomorphic to the graph $\text{EV}_1$ (if $y_1\neq y_2$)
or $Q_5$ (if $y_1=y_2$) from Figure~\ref{fig-basic}, and contains no $4$-face allowing a legal identification.
This contradicts the assumptions of Lemma~\ref{lemma-legalidentification}.

By symmetry, it follows that every edge of a $4$-face is incident with a special vertex.  In particular, the following holds.
\claim{cl-onlyone}{For every $4$-face $y_1y_2y_3y_4$ in $G_i$, there exists $k\in\{1,2\}$ such that both $y_k$ and $y_{k+2}$
are special.}

Consider now a $4$-face $f=x_1x_2x_3x_4$ in $G_i$ such that $x_1$ is not special.  By (\ref{cl-onlyone}),
$x_2$ and $x_4$ are special.  If $x_2,x_4\in V(C_j)$ for some $j\in\{i-1,i\}$, then $x_1$ and $x_3$ are in different components of
$G_i-\{x_2,x_4\}$, which contradicts (\ref{cl-alltriang}).  Hence, we have a strengthening of (\ref{cl-onlyone}):
\claim{cl-onlyones}{For every $4$-face $y_1y_2y_3y_4$ in $G_i$, there exists $k\in\{1,2\}$ such that one of $y_k$ and $y_{k+2}$
belongs to $C_{i-1}$ and the other one to $C_i$.}

Let us now consider any $4$-face $f=x_1x_2x_3x_4$ in $G_i$ such that the identification of $x_1$ with $x_3$ is legal.
By symmetry, we can assume that $x_1$ is not special, and by (\ref{cl-onlyones}), we can assume that $x_2\in V(C_{i-1})$ and $x_4\in V(C_i)$.
Since $x_1$ is not contained in a triangle, (\ref{cl-onlyones}) implies that $x_2$ and $x_4$ are the only
special neighbors of $x_1$.

Let us now distinguish two subcases depending on whether $x_3$ is special or not.
Let us first consider the subcase that $x_3$ is not special.  By symmetry, $x_2$ and $x_4$ are the only special neighbors of $x_3$.
In this case, we claim that the graph $H$ constructed at the beginning of the proof satisfies $r(H,\mathcal{C}')\ge r(G,\mathcal{C})$,
and thus Lemma~\ref{lemma-legalidentification} holds.  By (\ref{cl-alltriang}), the identification creates a new triangle
$T$.  Since $x_1$ and $x_3$ have no special neighbors other than $x_2$ and $x_4$,
even after the collapse of the triangles during the construction of $H''$, $T$ is vertex-disjoint with $C_{i-1}$ and $C_i$.
Hence, $H$ is obtained from $G$ by replacing $G_i$ with two subgraphs $H_j$ and $H_{j+1}$ for some $j\in\{1,\ldots, n'-1\}$
such that $H_j$ and $H_{j+1}$ intersect in $T$.  Clearly, $r(H,\mathcal{C}')\ge r(G,\mathcal{C})$, unless both $H_j$ and $H_{j+1}$
are quadrangulations and $G_i$ is not a quadrangulation.  As we argued before, every $4$-face of $H_j\cup H_{j+1}$ corresponds to either a
$4$-face in $G$, or to a contractible walk of length $6$ bounding a quadrangulated disk in $G$.  Hence, all the $(\ge\!5)$-faces of $G$
are destroyed by collapsing the triangles, that is, $G_i$ contains two paths $x_1y_1y_2x_3$ and $x_1y'_1y'_2x_3$ such that the open disk $\Lambda$ bounded by
the closed walk $x_1y_1y_2x_3y'_2y'_1$ contains a face of $G_i$ of length greater than $4$, and all faces of $G_i$ not contained
in $\Lambda$ have length $4$.  But then the edge $y_1y_2$ is contained in a $4$-face in $G_i$.  Since $y_1$ and $y_2$ are neighbors of $x_1$ and $x_3$,
respectively, they are not special.  This contradicts (\ref{cl-onlyone}).

Finally, we consider the subcase that $x_3$ is special.
By symmetry, we can assume that $x_3\in V(C_i)$.  If $C_i$ is a triangle, then since the identification of $x_1$ and $x_3$ is legal,
we conclude that no path $x_1y_1y_2x_3$ intersects $C_{i-1}$, and by the same argument as in the previous paragraph,
we show that $r(H,\mathcal{C}')\ge r(G,\mathcal{C})$.  Therefore, suppose that $|C_i|=4$.

By (\ref{cl-alltriang}), $G_i$ contains a path $x_1y_1y_2x_3$ disjoint from $\{x_2,x_4\}$.
Since the identification of $x_1$ with $x_3$ is legal, we have $y_2\not\in V(C_{i-1})$.
If $y_2\not\in V(C_i)$, then let $W$ be the contractible closed walk of length $7$
consisting of $x_4x_1y_1y_2x_3$ and the path of length three between $x_3$ and $x_4$ in $C_i$.
If $y_2\in V(C_i)$, then let $W$ be the contractible closed walk of length $5$
consisting of $x_4x_1y_1y_2$ and a path of length two between $y_2$ and $x_4$ in $C_i$.
In both cases, let $\Lambda$ be the open
disk bounded by $W$.  Since $y_2\not\in V(C_{i-1})$, (\ref{cl-onlyones}) implies that $\Lambda$ does not contain any $4$-face of $G_i$,
and by Theorem~\ref{thm-girth5}, we conclude that $\Lambda$ is a face.  Since this holds for every path of length three between $x_1$ and $x_3$,
it follows that $x_1y_1y_2x_3$ is the only such path, and thus $H'$ contains only one new triangle $T=xy_1y_2$.
Let $H_j$ be the subgraph of $H$ between $C_{i-1}$ and $T$.  Note that $H_j$ is not a quadrangulation, since otherwise
$G_i$ would contain a $4$-cycle incident with $x_1$ and $y_1$, contrary to (\ref{cl-onlyone}).
Consider the other subgraph $H_{j+1}$ in the chain that contains $T$.  Since $\Lambda$ is a face, it follows that
$y_1$ has degree two in $H_{j+1}$, and since $H$ is tame, the face of $H_{j+1}$ incident with $y_1$ has length greater than $4$
and $H_{j+1}$ is not a quadrangulation.  Therefore, $r(H,\mathcal{C}')\ge r(G,\mathcal{C})$ as required.
\end{proof}

Now, let us show that legal identification is always possible in a non-basic graph.

\begin{lemma}\label{lemma-4face}
Let $G$ be a tame critical graph embedded in the cylinder with rings $C_1$ and $C_2$ of length at most $4$,
such that every triangle of $G$ is equal to one of the rings.
Assume furthermore that the rings of $G$ are either vertex-disjoint, or one of them has length $3$ and the other
one length $4$.  If $G$ has no $4$-face admiting a legal identification (with respect to the sequence $C_1,C_2$ of cutting cycles),
then $G$ is basic.
\end{lemma}
\begin{proof}
Since no $4$-face of $G$ admits a legal identification,
\claim{cl-allp}{for any $4$-face $x_1x_2x_3x_4$ of $G$ such that $x_1\not\in V(C_1\cup C_2)$,
there exists a path of length three between $x_1$ and $x_3$ disjoint with $\{x_2,x_4\}$.}

Suppose that $G$ contains a $4$-face $x_1x_2x_3x_4$ such that $x_3,x_4\not\in V(C_1\cup C_2)$.
By (\ref{cl-allp}), $G$ contains paths of length three between $x_1$ and $x_3$, and between $x_2$ and $x_4$,
and the two paths must intersect, forming a triangle.  This triangle is equal to say $C_1$ and contains $x_1$ and $x_2$.
Hence, $G$ contains paths $x_1zy_1x_3$ and $x_2zy_2x_4$, where $C_1=x_1x_2z$.  At least one of the $4$-cycles $zx_2x_3y_1$ and $zx_1x_4y_2$ (say the former) is contractible.
By Lemma~\ref{lemma-fr}, $zx_2x_3y_1$ is a face.  Note that since $G$ is tame and $x_3,x_4\not\in V(C_1\cup C_2)$,
$y_1$ is not adjacent to $y_2$ or $x_4$.  Consequently, $G$ contains no path of length three between $y_1$ and $x_2$
disjoint with $\{z,x_3\}$, and by (\ref{cl-allp}), it follows that $y_1\in V(C_2)$.  Therefore, the $4$-cycle
$zx_1x_4y_2$ is also contractible, bounds a face, and by symmetry, $y_2\in V(C_2)$.  However,
then $G$ is isomorphic to the basic graph $\text{Q}_5$ or $\text{EV}_1$.

Hence, we can assume that each edge of a $4$-face of $G$ is incident with a vertex of $C_1\cup C_2$, or
equivalently
\claim{cl-oone}{for every $4$-face $y_1y_2y_3y_4$ in $G$, there exists $k\in\{1,2\}$ such that both $y_k$ and $y_{k+2}$
belong to $V(C_1\cup C_2)$.}
Suppose that $x_1x_2x_3x_4$ is a $4$-face in $G$ such that $x_1\not\in V(C_1\cup C_2)$.  By (\ref{cl-oone}),
$x_2,x_4\in V(C_1\cup C_2)$.  If $x_2,x_4\in V(C_j)$ for some $j\in\{1,2\}$, then $x_1$ and $x_3$ are in different components of
$G-\{x_2,x_4\}$, which contradicts (\ref{cl-allp}).  Hence, we have a strengthening of (\ref{cl-oone}):
\claim{cl-oones}{For every $4$-face $y_1y_2y_3y_4$ in $G_i$, there exists $k\in\{1,2\}$ such that one of $y_k$ and $y_{k+2}$
belongs to $C_1$ and the other one to $C_2$.}

Suppose that $x_1x_2x_3x_4$ is a $4$-face such that $x_1\not\in V(C_1\cup C_2)$.
By symmetry, we can assume that $x_2\in V(C_1)$ and $x_4\in V(C_2)$.
Since $x_1$ is not contained in a triangle, (\ref{cl-oones}) implies that $x_2$ and $x_4$ are the only
neighbors of $x_1$ in $V(C_1\cup C_2)$.

If $x_3$ is not special, then by symmetry, $x_2$ and $x_4$ are the only neighbors of $x_3$ in $V(C_1\cup C_2)$.
However, then no path $x_1z_1z_2x_3$ intersects $C_1$ or $C_2$,
and thus the identification of $x_1$ and $x_3$ is legal, which is a contradiction.

Hence,
\claim{cl-all3}{every $4$-face of $G$ intersects $C_1\cup C_2$ in at least three vertices.}
Suppose that $C_1$ and $C_2$ are not disjoint, and thus say $|C_1|=3$ and $|C_2|=4$.  If $C_1$ and $C_2$ share an edge,
then Lemma~\ref{lemma-fr} implies that $G=C_1\cup C_2$ has no $4$-face, and thus it is basic.  Let us consider the
case that $C_1$ and $C_2$ share only one vertex.  If $G$ contains an edge $e\not\in E(C_1\cup C_2)$ with both ends in $V(C_1\cup C_2)$,
then since $G$ is tame, we conclude that $G$ is the basic graph $J_1$.  If $G$ contains no such edge and contains a $4$-face,
then by (\ref{cl-all3}), such a $4$-face shares two edges with $C_1\cup C_2$.  However, then Theorem~\ref{thm-7cyc}
gives a contradiction with (\ref{cl-oones}) or (\ref{cl-all3}).

Therefore, we can assume that $C_1$ and $C_2$ are vertex-disjoint.
Suppose that $x_1x_2x_3x_4$ is a $4$-face such that $x_1\not\in V(C_1\cup C_2)$.
By (\ref{cl-oones}) and (\ref{cl-all3}), we can assume that $x_2\in V(C_1)$ and $x_3,x_4\in V(C_2)$.  
Recall that $x_2$ and $x_4$ are the only neighbors of $x_1$ in $V(C_1\cup C_2)$.
Since the identification of $x_1$ and $x_3$ is not legal, there exists a path
$x_1y_1y_2x_3$ disjoint with $\{x_2,x_4\}$ such that $y_2\in V(C_1)$.  Since $x_2\in C_1$, $x_3\in V(C_2)$, $C_1$ and $C_2$ are vertex-disjoint,
and $G$ contains no non-ring triangles, it follows that $x_2x_3$ is not contained in a triangle.  It follows that
$y_2$ is not adjacent to $x_2$, and thus $|C_1|=4$.  Let $C_1=x_2z_1y_2z_2$, with the labels chosen so that
$x_2x_3y_2z_1$ is a contractible $4$-cycle.  By Lemma~\ref{lemma-fr}, $x_2x_3y_2z_1$ and $x_2z_2y_2y_1x_1$ are
faces.  Also by Lemma~\ref{lemma-fr}, $y_1$ has degree at least three.  By (\ref{cl-all3}), the edge
$x_1y_1$ is not incident with a $4$-face, and thus by Theorems~\ref{thm-6cyc} and \ref{thm-7cyc} applied to the cycle consisting
of $x_4x_1y_1y_2x_3$ and the path of length $|C_2|-1$ in $C_2$ between $x_3$ and $x_4$ implies that $|C_2|=4$ and the edge
$y_1y_2$ is incident with a $4$-face.  By (\ref{cl-all3}), $y_1$ has a neighbor in $C_2$,
and thus $G$ is isomorphic to the basic graph $\text{EV}_2$.

It follows that every $4$-face in $G$ has all vertices contained in $V(C_1\cup C_2)$.  
Note that since $G$ is tame, $C_1$ and $C_2$ are induced cycles.
If $G$ has no $4$-face, then by Lemma~\ref{lemma-fr} it contains no contractible $4$-cycle, and thus it is basic.
Hence, it suffices to consider the case that $G$ contains a $4$-face $f=x_1x_2x_3x_4$.  
By symmetry, we can assume that $x_1,x_2\in V(C_1)$
and $x_4\in V(C_2)$.  Let $h$ be the face of $C_1\cup C_2\cup \{x_2x_3,x_3x_4, x_4x_1\}$ distinct from $f$ and the rings.
Note that $|h|=|C_1|+|C_2|\le 8$.  If $G$ has no face of length $4$ other than $f$, then by Theorem~\ref{thm-girth5} applied to the
subgraph of $h$ drawn in the closure of $h$, we conclude that either $h$ is a face of $G$ and $G$ is one of the basic graphs
$\text{Tr}'_1$, $\text{Tr}'_2$, $\text{Xq}_5$, $\text{Xq}_6$, or $\text{Xq}_7$; or $|C_1|=|C_2|=4$ and one edge of $G$ is drawn in $h$,
and $G$ is a basic graph $\text{S}_1$ or $\text{S}_2$.

Let us consider the case that $G$ has a face $f'$ of length $4$ distinct from $f$; hence, an edge of $G$ splits $h$ into a $4$-face $f'$ and a $(|C_1|+|C_2|-2)$-face $h'$.
If $G$ contains no $4$-face other than $f$ and $f'$, then by Theorem~\ref{thm-girth5}, $h'$ is a face of $G$, and thus
$G$ is one of the basic graphs $\text{Xq}_1$, $\text{Xq}_2$, $\text{Xq}_3$, $\text{Xq}_4$, $\text{Q}_5$, $\text{Tr}_1$, or $\text{Tr}_2$.
Finally, $G$ might also have four $4$-faces, and then $G$ is one of the basic graphs
$\text{Q}_1$, $\text{Q}_2$, $\text{Q}_3$, or $\text{Q}_4$.
\end{proof}

By combining Lemmas~\ref{lemma-legalidentification} and Lemma~\ref{lemma-4face}, we obtain the following.

\begin{lemma}\label{lemma-basic}
Let $G$ be a tame graph embedded in the cylinder with rings of length at most $4$.  Suppose that $G$ is a chain of
$n\ge 1$ critical graphs with the sequence $\mathcal{C}$ of cutting cycles.
Then there exists a tame graph $G'$ embedded in the cylinder with the same rings, such that $G'$ dominates $G$ and
$G'$ is a chain of at least $n$ basic graphs with cutting cycles $\mathcal{C}'$, with $r(G',\mathcal{C}')\ge r(G,\mathcal{C})$.
\end{lemma}
\begin{proof}
We prove the claim by induction on $|V(G)|$.
Let $G$ be a chain of critical graphs $G_1$, \ldots, $G_n$ with cutting cycles $\mathcal{C}=\{C_0,\ldots,C_n\}$.
If all of $G_1$, \ldots, $G_n$ are basic, then the claim is trivialy true with $G'=G$.
Otherwise, by Lemma~\ref{lemma-4face}, there exists $i\in\{1,\ldots, n\}$ such that a $4$-face in $G_i$
admits a legal identification.  Then, the claim follows by the induction hypothesis applied to
the graph $H$ obtained by Lemma~\ref{lemma-legalidentification}.
\end{proof}

Dvo\v{r}\'ak and Lidick\'y~\cite{dvolid} gave an exact description of critical graphs embedded in the cylinder with rings of length at most $4$ and without
contractible $(\le\!4)$-cycles---in addition to the infinite family of reduced Thomas-Walls graphs, there are only 95 such critical graphs.
In particular, their examination gives the following.
\begin{theorem}[Dvo\v{r}\'ak and Lidick\'y~\cite{dvolid}]\label{thm-dvolid}
Let $G$ be a graph embedded in the cylinder with rings of length at most $4$, such that $G$ contains no contractible $(\le\!4)$-cycles.
If $G$ is critical and the distance between its rings is at least $6$, then $G$ is a reduced Thomas-Walls graph.
\end{theorem}

We need the following observation about basic graphs.
\begin{lemma}\label{lemma-simno4}
Let $G$ be a tame chain of basic graphs, such that each of them either contains no contractible $4$-cycles,
or is isomorphic to $\text{EV}_1$, $\text{EV}_2$, $\text{Tr}_1$, $\text{Tr}_2$, $\text{S}_1$, or $\text{S}_2$,
with pairwise vertex-disjoint cutting cycles $C_0$, \ldots, $C_n$.
If $n\ge 5$, then $G$ is dominated by a graph $G'$ with the same vertex-disjoint rings $C_0$ and $C_n$
and without contractible $(\le\!4)$-cycles.
Furthermore, if $|C_0|=4$ and $e_0$ is any edge of $C_0$, then $G'$ can be chosen so that $e$ is not contained in a triangle in $G'$.
\end{lemma}
\begin{proof}
Let $C_0$, \ldots, $C_n$ be the cutting cycles of the chain, and for $i=1,\ldots,n$, let $G_i$ denote the subgraph of $G$
drawn between $C_{i-1}$ and $C_i$.  

We start with several observations about graphs dominating $\text{EV}_1$, $\text{EV}_2$, $\text{Tr}_1$, $\text{Tr}_2$, $\text{S}_1$, or $\text{S}_2$.
Let $H$ be one of these graphs, with rings $u_1u_2\ldots$ and $v_1v_2\ldots$ labelled as in Figure~\ref{fig-basic}.
In all the identifications described below, we suppress arising parallel edges.
\begin{itemize}
\item If $H=\text{EV}_1$ or $H=\text{EV}_2$, then
\begin{itemize}
\item let $H'$ be the graph isomorphic to one of the graphs I$_1$ or I$_2$ depicted in Figure~\ref{fig-identification}, obtained from $H$ by identifying $u_2$ with $v_2$ and $u_1$ with $v_1$ and by removing
the vertices that do not belong to the rings, and
\item let $H''$ be the graph isomorphic to one of the graphs I$_1$ or I$_2$, obtained from $H$ by identifying $u_2$ with $v_2$ and $u_3$ with $v_3$ and by removing
the vertices that do not belong to the rings.
\end{itemize}
Let $\mathcal{R}(H)=\{H',H''\}$ and $\mathcal{D}(H)=\{u_2,v_2\}$.
\item If $H=\text{Tr}_1$, then
\begin{itemize}
\item let $H'$ be the graph isomorphic to the graph I$_1$ depicted in Figure~\ref{fig-identification}, obtained from $H$ by identifying $u_2$ with $v_2$ and $u_1$ with $v_1$, and
\item let $H''$ be the graph isomorphic to the graph I$_1$, obtained from $H$ by identifying $u_2$ with $v_4$ and $u_3$ with $v_1$.
\end{itemize}
Let $\mathcal{R}(H)=\{H',H''\}$ and $\mathcal{D}(H)=\{u_2,v_1\}$.
\item If $H=\text{Tr}_2$, then
\begin{itemize}
\item let $H'$ be the graph isomorphic to the graph I$_1$ depicted in Figure~\ref{fig-identification}, obtained from $H$ by identifying $u_1$ with $v_1$ and $u_2$ with $v_2$, and
\item let $H''$ be the graph isomorphic to the graph I$_4$ depicted in Figure~\ref{fig-identification}, obtained from $H$ by identifying $u_1$ with $v_3$. 
\end{itemize}
Let $\mathcal{R}(H)=\{H',H''\}$ and $\mathcal{D}(H)=\{u_1\}$.
\item If $H=\text{S}_1$, then
let $H'$ be the graph isomorphic to the graph I$_3$ depicted in Figure~\ref{fig-identification}, obtained from $H$ by identifying $u_1$ with $v_1$.
Let $\mathcal{R}(H)=\{H'\}$ and $\mathcal{D}(H)=\{u_1,v_1\}$.
\item If $H=\text{S}_2$, then
let $H'$ be the graph isomorphic to the graph I$_3$ depicted in Figure~\ref{fig-identification}, obtained from $H$ by identifying $u_4$ with $v_4$.
Let $\mathcal{R}(H)=\{H'\}$ and $\mathcal{D}(H)=\{u_4,v_4\}$.
\end{itemize}
Observe that every graph in $\mathcal{R}(H)$ dominates $H$ and contains no contractible $(\le\!4)$-cycle.  For every edge $e$ of a ring of $H$, there exists
at least one graph $H_e\in\mathcal{R}(H)$ such that the corresponding edge in $H_e$ is not shared by both rings.  Furthermore,
if a vertex $v$ of a ring of $H$ does not belong to $\mathcal{D}(H)$, then there exists at least one graph $H_v\in\mathcal{R}(H)$ such that the corresponding vertex in $H_v$ is not shared by both rings.

\begin{figure}[ht]
\begin{center}
\includegraphics[page=1,scale=1]{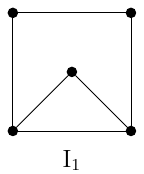}
\includegraphics[page=2,scale=1]{fig-identification}
\includegraphics[page=3,scale=1]{fig-identification}
\includegraphics[page=22,scale=1]{fig-basic}
\end{center}
\caption{Results of identifications in Lemma~\ref{lemma-simno4}.}\label{fig-identification}
\end{figure}

For $i=1,\ldots, n$, let $G'_i$ be the graph obtained as follows.
Let $v_i$ be a vertex of $C_{i-1}$ not belonging
to $\mathcal{D}(G_i)$, chosen so that
\begin{itemize}
\item if $i=1$, then $v_i$ is incident with $e_0$,
\item if $i\ge 2$ and the intersection $Z_i$ of $C_0$ and $C_{i-1}$ in $G'_1\cup\ldots\cup G'_{i-1}$ is not a subset of $\mathcal{D}(G_i)$, then $v_i\in Z_i\setminus\mathcal{D}(G_i)$, and
\item if $i\ge 2$, $|C_{i-1}|=4$ and the intersection $Y_i$ of $C_{i-2}$ and $C_{i-1}$ in $G'_{i-1}$ is not a subset of $\mathcal{D}(G_i)$, then $v_i\in Y_i\setminus\mathcal{D}(G_i)$.
\end{itemize}
If $G_i$ does not contain any contractible $4$-cycle, then let $G'_i=G_i$.  Otherwise, let $G'_i$ be an element of $\mathcal{R}(G_i)$, chosen so that
the vertex of $G'_i$ corresponding to $v_i$ is not shared by both rings of $G'_i$.

Note that this ensures that no two triangles of $G'=G'_1\cup\ldots\cup G'_n$ share an edge, and thus $G'$ contains no contractible $(\le\!4)$-cycles.
Furthermore, the edge corresponding to $e_0$ is not contained in a triangle of $G'$.

The graph $G'$ dominates $G$.  Note that $|Z_2|\le 2$ and $|Z_i|\le 1$ for $i\ge 3$, and thus
the rings of $G'$ share at most one vertex.
If the rings of $G'$ are vertex-disjoint, then $G'$ satisfies the conclusion of the lemma.
Hence, suppose the rings of $G'$ share a unique vertex $z$.  This is only possible if each of $G_1$, \ldots, $G_n$ is
isomorphic to $\text{EV}_1$, $\text{EV}_2$, $\text{Tr}_1$, $\text{Tr}_2$, $\text{S}_1$, or $\text{S}_2$.
For $i=0,\ldots, n$, let $z_i$ denote the vertex of $C_i$ corresponding to $z$.

Consider the graph $F_1=G'_1\cup G'_2\cup G'_3$, whose rings $C_0$ and $C_3$ intersect only in $z$.  
Observe that every edge of $F_1$ with both ends in $V(C_0\cup C_3)$ is contained in $E(C_0\cup C_3)$.  Since $|C_0|+|C_3|\le 8$ and $F_1$ contains
no contractible $(\le\!4)$-cycle, Theorem~\ref{thm-girth5} shows that every precoloring of $C_0\cup C_3$ that assigns the same color to $z_0$ and $z_3$
extends to a $3$-coloring of $F_1$.  Since $z$ is contained in both rings of $G'$, the choice of $G'_4$ implies that $z_3$ belongs to $\mathcal{D}(G_4)$.
Hence, $F_1\cup G_4$ is dominated by a graph obtained from $G_4$ by adding the cycle $C_0$ and identifying $z_0$ with $z_3$.
By inspecting all choices of $G_4\in\{\text{EV}_1,\text{EV}_2,\text{Tr}_1,\text{Tr}_2,\text{S}_1,\text{S}_2\}$ and $z_3\in \mathcal{D}(G_4)$, we conclude
that the following holds.
\begin{itemize}
\item if $|C_4|=3$, then $F_1\cup G_4$ is dominated by the graph $F_2$ obtained from the disjoint union of $C_0$ and $C_4$ by adding an edge
between $z_0$ and a vertex of $C_4$, and
\item if $|C_4|=4$, then either
\begin{itemize}
\item $F_1\cup G_4$ is dominated by the graph $F'_2$ isomorphic to $\text{Xq}_6$ or $\text{Tr}'_2$ obtained from the disjoint union of $C_0$ and $C_4$ by adding edges
between $z_0$ and two non-adjacent vertices of $C_4$ distinct from $z_4$, or
\item $G_4$ is isomorphic to $\text{Tr}_1$.
\end{itemize}
\end{itemize}
If $|C_4|=3$, then $F_2\cup G'_5\cup \ldots\cup G'_n$ satisfies the conclusions of the lemma.  Let us consider the case that $|C_4|=4$.
If $G_4$ is isomorphic to $\text{Tr}_1$ (and thus $|C_3|=3$), note that none of the conditions for specifying $v_4$ applies, and thus
we can choose $G'_4$ so that $z$ does not correspond to any vertex of $\mathcal{D}(G_5)$.  However, then the choice of $v_5$ ensures
that $C_0$ and $C_5$ are vertex-disjoint in $G'$, which is a contradiction.

Hence, we can assume that $F_1\cup G_4$ is dominated by the graph $F'_2$.
By inspection of all choices of $G_5\in\{\text{EV}_1,\text{EV}_2,\text{Tr}_1,\text{Tr}_2,\text{S}_1,\text{S}_2\}$ and $z_4\in \mathcal{D}(G_4)$,
we conclude that $F'_2\cup G_5$ is dominated by a graph $F_3$
obtained from the vertex-disjoint union of cycles $C_0$ and $C_5$ by adding an edge between them,
and thus $F_3\cup G'_6\cup\ldots \cup G'_n$ satisfies the conclusions of the lemma.
\end{proof}

Let us now prove a variation on Theorem~\ref{thm-dvolid}.

\begin{lemma}\label{lemma-simchain}
Let $G$ be a tame chain of $n$ basic graphs, such that each of them either contains no contractible $4$-cycles,
or is isomorphic to $\text{J}_1$, $\text{EV}_1$, $\text{EV}_2$, $\text{Tr}_1$, $\text{Tr}_2$, $\text{S}_1$, or $\text{S}_2$.
If $n\ge 32$, then either every precoloring of the rings of $G$ extends to a $3$-coloring of $G$, or both rings of $G$
have length $4$ and $G$ is dominated by the reduced Thomas-Walls graph $T'_4$.
\end{lemma}
\begin{proof}
Let $C_0$, \ldots, $C_n$ be the cutting cycles of the chain.
Let $s_0=0$ if $C_0$ and $C_1$ are vertex-disjoint, and $s_0=1$ otherwise.
Let $s_1=n$ if $C_{n-1}$ and $C_n$ are vertex-disjoint, and $s_1=n-1$ otherwise.
Let $F$ be the union of subgraphs of $G$ drawn between $C_0$ and $C_{s_0}$, and between $C_{s_1}$ and $C_n$.
Note that if a component of $F$ contains a contractible $4$-cycle, then it is isomporphic to
$\text{J}_1$.  Identify two opposite vertices in each $4$-face of $F$, obtaining a graph $F'$.

For $i=1,\ldots, 5$, let $G_i$ denote the subgraph of $G$ drawn between $C_{s_0+5i-5}$ and $C_{s_0+5i}$.  Let $G_6$ denote the subgraph
of $G_i$ drawn between $C_{s_0+25}$ and $C_{s_1}$.  
By the definition of a chain, $\text{J}_1$ does not appear in the subchain bounded by $C_{s_0}$  and $C_{s_1}$.
For $i=1,\ldots, 6$, Lemma~\ref{lemma-simno4} implies that $G_i$ is dominated by
a graph $G'_i$ with the same vertex-disjoint rings, such that the graph $G'=F'\cup G'_1\cup\ldots\cup G'_6$ contains no
contractible $(\le\!4)$-cycles---this is ensured by choosing $e_0$ to be an edge of $C_{s_0+5i-5}$ contained in a triangle in $G'_{i-1}$ (if any) for $i=2,\ldots, 6$.

Note that the distance between $C_0$ and $C_n$ in $G'$ is
at least $6$.  Suppose that not every precoloring of $C_0\cup C_n$ extends to a $3$-coloring of $G$.  Then $G'$ is dominated
by its maximal critical subgraph $G''$, and by Theorem~\ref{thm-dvolid}, $G''$ is a reduced Thomas-Walls graph in that the
distance between $C_0$ and $C_n$ is at least $6$.  By Lemma~\ref{lemma-col-tw}, $G$ is dominated by the reduced Thomas-Walls graph $T'_4$.
\end{proof}

We call graphs $\text{Xq}_1$, \ldots, $\text{Xq}_7$ \emph{almost quadrangulations}.
We need the following claim.
\begin{lemma}\label{lemma-almostq}
Let $G$ be a tame chain of $n$ basic graphs, such that each of them 
is an almost quadrangulation.
If $n\ge 4$, then every precoloring of the rings of $G$ extends to a $3$-coloring of $G$.
\end{lemma}
\begin{proof}
Let $C_0$, \ldots, $C_n$ be the cutting cycles of the chain, and let $\psi$ be a $3$-coloring of $C_0\cup C_n$.
The graph $G$ is dominated by a quadrangulation of the cylinder with rings at distance at least $4$ from each other.
Hence, if $|C_0|=4$, then $\psi$ extends to a $3$-coloring of $G$ by Lemma~\ref{lemma-col44}.
If $|C_0|=3$, then $G$ is a chain of the copies of $\text{Xq}_7$.  Observe that $\psi$ extends to a $3$-coloring $\varphi$
of the subgraph of $G$ drawn between $C_0$ and $C_1$ such that the winding number of $\varphi$ on $C_1$ is
the same as the winding number of $\psi$ on $C_n$. By Lemma~\ref{lemma-col44}, it follows that $\psi$ extends to a $3$-coloring of $G$.
\end{proof}

By combining these results, we obtain the following. 
\begin{lemma}\label{lemma-domin}
Let $G$ be a tame graph embedded in the cylinder with rings of length at most $4$.  
If $G$ is a chain of graphs, at least $131$ of which are not quadrangulated,
then either every precoloring of the rings of $G$ extends to a $3$-coloring of $G$, or both rings of $G$ have length $4$ and $G$ is dominated by the reduced Thomas-Walls graph $T'_4$.
\end{lemma}
\begin{proof}
Let $\mathcal{C}=C_0,C_1,\ldots, C_n$ be the sequence of cutting cycles of the chain $G$, and for $i=1,\ldots, n$, let $G_i$ be the subgraph of $G$ drawn between $G_{i-1}$ and $G_i$.
Without loss of generality, $G_i$ is critical, as otherwise we can replace $G_i$ by its maximal critical subgraph.
Furthermore, by Lemma~\ref{lemma-basic}, we can assume that $G_i$ is basic.
If $G_i$ is a quadrangulation, then $G_i$ is one of $Q_1$, \ldots, $Q_5$; in this case, we identify 
the vertices of the rings of $G_i$ as indicated by dotted lines in Figure~\ref{fig-basic}.

Hence, assume that none of $G_1$, \ldots, $G_n$ is a quadrangulation.  By the assumptions of the lemma, we have $n\ge 131$.

Suppose that there exists $i\in\{1, \ldots, n-3\}$ such that $G_i$, \ldots, $G_{i+3}$ are almost quadrangulations.
By Theorem~\ref{thm-onetri}, any $3$-coloring $\psi$ of $C_0\cup C_n$ extends to a $3$-coloring of $G_1\cup\ldots\cup G_{i-1}\cup G_{i+4}\cup\ldots\cup G_n$,
and by Lemma~\ref{lemma-almostq}, $\psi$ extends to a $3$-coloring of $G$.
Therefore, we can assume that no four consecutive graphs in the chain $G'''$ are almost quadrangulations.

Extend all almost quadrangulations in $G$ to quadrangulations Q$_1$, \ldots, Q$_5$ and 
identify the vertices of their rings as indicated by dotted lines in Figure~\ref{fig-basic}.
The resulting graph $G'$ is a chain of at least $32$ basic graphs dominating $G$,
such that none of the graphs in the chain is a quadrangulation or almost quadrangulation.
Also, if $\text{J}_1$ appears in the chain (necessarily as the first or the last element), identify the opposite vertices of its $4$-face not contained in the ring of $G'$,
thus turning it into $\text{I}_1$.
Similarly, we turn each of the graphs in the chain that is isomorphic to $\text{Tr}'_1$ or $\text{Tr}'_2$ into
$\text{Tr}_1$ or $\text{Tr}_2$ by adding an edge.
The claim of Lemma~\ref{lemma-domin} then follows by Lemma~\ref{lemma-simchain}.
\end{proof}

Actually, we only need the following consequence of Lemmas~\ref{lemma-col-tw} and \ref{lemma-domin}

\begin{corollary}\label{cor-chaincol}
Let $G$ be a tame graph embedded in the cylinder with rings $C_1$ and $C_2$ of length at most $4$.
Suppose that $G$ is a chain of graphs, at least $131$ of which are not quadrangulated.
If $|C_1|=3$ or $|C_2|=3$, then every precoloring of $C_1\cup C_2$ extends to a $3$-coloring of $G$.
If $|C_1|=|C_2|=4$, then for every $3$-coloring $\psi_1$ of $C_1$, there exists $v\in V(C_2)$
such that for every $v$-diagonal $3$-coloring $\psi_2$ of $C_2$, the precoloring $\psi_1\cup\psi_2$
extends to a $3$-coloring of $G$.
\end{corollary}

Let us remark that using computer, we can signficantly improve the bound of Corollary~\ref{cor-chaincol} as follows.
\begin{lemma}\label{lemma-chaincol}
Let $G$ be a tame graph embedded in the cylinder with rings $C_1$ and $C_2$ of length at most $4$.
Suppose that $G$ is a chain of $n\ge 7$ graphs, at least $3$ of which are not quadrangulated.
If $|C_1|=3$ or $|C_2|=3$, then every precoloring of $C_1\cup C_2$ extends to a $3$-coloring of $G$.
If $|C_1|=|C_2|=4$, then for every $3$-coloring $\psi_1$ of $C_1$, there exists $v\in V(C_2)$
such that for every $v$-diagonal $3$-coloring $\psi_2$ of $C_2$, the precoloring $\psi_1\cup\psi_2$
extends to a $3$-coloring of $G$.
\end{lemma}
\begin{proof}
By Lemma~\ref{lemma-basic}, we can assume that all the graphs in the chain are basic.  Furthermore, without loss of generality,
they are critical.  We consider the chains consisting of copies of the graphs depicted in Figure~\ref{fig-basic}, the graph I$_4$, and the
graphs from \cite{dvolid} whose rings are disjoint and not separated by another triangle or disjoint $4$-cycle.  By computer enumeration, we verify that the claim holds for $n=7$.

For longer chains, we prove the claim by induction on $n$.  By symmetry, we can assume that if $|C_1|=3$, then $|C_2|=3$.
If any of the graphs in the chain is isomorphic to a quadrangulation Q$_1$, \ldots, or Q$_5$,
then identify the vertices of their rings as indicated by dotted lines in Figure~\ref{fig-basic}, and apply induction to the resulting graph.
Otherwise, let $K$ be the first cutting cycle of the chain distinct from $C_1$.  We extend the $3$-coloring of $C_1$ to the subgraph of $G$
drawn between $C_1$ and $K$ by Theorem~\ref{thm-onetri}, and then the claim follows by the induction hypothesis for the subgraph of $G$ drawn between $K$ and $C_2$.
\end{proof}

Let us mention that Lemma~\ref{lemma-chaincol} is not true if we demand just two non-quadrangulations or if the chain has length at most 6.

\section{Disk with two triangles}\label{sec-crtri}

In this section, we consider the case of the graphs embedded in the disk with a ring of length $4$ and with exactly two triangles,
and we aim to prove Theorem~\ref{thm-crtri}.

Borodin et al.~\cite{4c4t} described all non-$3$-colorable planar graphs with exactly $4$ triangles.  The following is a special case of their
main result relevant to us.
\begin{theorem}[Borodin et al.~\cite{4c4t}]\label{thm-4c4t}
If $G$ is a $4$-critical planar graph with exactly four triangles and there exists an edge $e\in E(G)$ intersecting two of them,
then $G-e$ is a patched Havel-Thomas-Walls graph with the interface pair $e$.
\end{theorem}

As a corollary, we have the following special case of Theorem~\ref{thm-crtri}.

\begin{corollary}\label{cor-patched}
Let $G$ be a graph embedded in the disk with at most two triangles and with ring $C=v_1v_2v_3v_4$.
If $G$ is critical and has no $3$-coloring that is $v_1$-diagonal on $C$,
then $G$ is obtained from a patched Havel-Thomas-Walls graph by framing on its interface pair $v_1v_3$.
\end{corollary}
\begin{proof}
Since $G$ is critical, Theorem~\ref{thm-onetri} implies that any non-facial $4$-cycle in $G$ separates
the hole of the disk from both triangles of $G$.

Let $C'$ be a non-facial $4$-cycle in $G$ containing $v_1$ and $v_3$, such that the subgraph of $G$ drawn between $C'$
and $C$ is maximal. Let $G'=G-(V(C)\setminus V(C'))$.
Since the $4$-cycles in $C\cup C'$ distinct from $C$ and $C'$ do not separate the hole from the triangles,
they bound faces, and thus $G$ is obtained from $G'$ by framing on $v_1v_3$.

Let $G''=G'+v_1v_3$.  Note that $G''$ is not $3$-colorable, and by the choice of $C'$, it contains exactly
four triangles.  By Theorem~\ref{thm-3t}, $G''$ contains a $4$-critical subgraph $H$ with all four triangles.
By Theorem~\ref{thm-4c4t}, $H-v_1v_3$ is a patched Havel-Thomas-Walls graph with the interface pair $v_1v_3$.
Since $G$ is critical, Lemma~\ref{lemma-fr} implies that every non-facial $(\le\!5)$-cycle of $G$ separates
the hole of the disk from at least one triangle of $G$.  However, all faces $H-v_1v_3$ have length
at most $5$, and thus $G''=H$.

It follows that $G$ is obtained from the patched Havel-Thomas-Walls graph $G'$ by framing on its interface pair $v_1v_3$.
\end{proof}

Finally, we need a result on the density of $4$-critical graphs by Kostochka and Yancey~\cite{koyan}.

\begin{theorem}\label{thm-crdens}
For every $n\ge 4$, every $4$-critical graph with $n$ vertices has at least $\frac{5n-2}{3}$ edges.
\end{theorem}

We can now prove the main theorem of this section.

\begin{proof}[Proof of Theorem~\ref{thm-crtri}]
We prove the claim by the induction on the number of vertices of $G$; hence, we assume that the claim holds
for all graphs with less than $|V(G)|$ vertices.

Let $C=v_1v_2v_3v_4$ be the ring of $G$.  Let $\psi_1$, $\psi_2$ and $\psi_3$ be a $v_1$-diagonal,
a $v_2$-diagonal, and a bichromatic $3$-coloring of $C$, respectively.
If either $\psi_1$ or $\psi_2$ does not extend to a $3$-coloring of $G$, then the claim follows from Corollary~\ref{cor-patched}.
Hence, we can assume that $\psi_1$ and $\psi_2$ extend to $3$-colorings $\varphi_1$ and $\varphi_2$ of $G$, respectively.
Since $G$ is critical, $\psi_3$ does not extend to a $3$-coloring of $G$.

Note that if say $v_1$ had degree $2$, then we could
recolor $v_1$ in the coloring $\varphi_1$ and obtain a $3$-coloring of $G$ whose restriction to $C$ is bichromatic,
which is a contradiction.  Similarly, we conclude that every vertex of $C$ has degree at least three.
Also, since $\varphi_1(v_2)=\varphi_1(v_4)$, the graph $G^\star_2$ obtained from $G$ by identifying $v_2$ with $v_4$
is $3$-colorable.  Symmetrically, the graph $G^\star_1$ obtained from $G$ by identifying $v_1$ with $v_3$ is $3$-colorable.

Suppose for a contradiction that $G$ has no non-ring $4$-faces.  Let $n$, $m$ and $s$ denote the number of vertices, edges and faces
of $G$, respectively.  Then $2m\ge 5(s-3)+4+2\cdot 3=5s-5$.  By Euler's formula, we have $s=m+2-n$, and thus
\begin{eqnarray*}
2m&\ge&5(m+2-n)-5=5m-5n+5\\
5n-5&\ge&3m
\end{eqnarray*}
and $m\le \frac{5n-5}{3}$.  Let $G'$ be the graph (not embedded in the disk) obtained from $G$ by identifying $v_1$ with $v_3$ to a vertex $z_1$
and $v_2$ with $v_4$ to a vertex $z_2$ and by suppressing parallel edges.  Note that $G'$ is not $3$-colorable, since $\psi_3$ does not extend to a $3$-coloring of $G$.
Let $G''$ be a $4$-critical subgraph of $G'$.  Since $G^\star_1$ and $G^\star_2$ are $3$-colorable, we have $z_1,z_2\in V(G'')$.
For every $v\in V(G')\setminus \{z_1,z_2\}$, note that $\psi_3$ extends to a $3$-coloring of $G-v$ by the criticality of $G$,
and thus $v\in V(G'')$.  Thus, $V(G'')=V(G')$.  Note that $|E(G')|=m-3$ and $|V(G')|=n-2$ since $C$ is replaced by the edge $z_1z_2$. Thus
$$|E(G'')|\le |E(G')|=m-3\le\frac{5n-14}{3}=\frac{5|V(G'')|-4}{3}.$$
This contradicts Theorem~\ref{thm-crdens}.

It follows that $G$ contains a $4$-face $x_1x_2x_3x_4$.  Since all vertices of $C$ have degree at least $3$, we can assume that
$x_1,x_2\not\in V(C)$.  If $x_3,x_4\in V(C)$, say $x_3=v_3$ and $x_4=v_4$, and $x_1$ is adjacent to $v_1$ and $x_2$ is adjacent to $v_2$,
then $G$ is a tent.  Hence, by symmetry, we can assume that either $x_3\not\in V(C)$, or $x_3=v_i$ for some $i\in \{3,4\}$ and
$x_1$ is not adjacent to $v_{i-2}$.
Let $G_0$ be the graph obtained from $G$ by identifying $x_1$ with $x_3$ to a new vertex $x$.  Note that $C$ is the ring of $G_0$
and $C$ is an induced cycle.
Observe that every $3$-coloring of $G_0$ corresponds to a $3$-coloring of $G$, obtained by giving $x_1$ and $x_3$ the color of $x$.
Consequently, $\psi_3$ does not extend to a $3$-coloring of $G_0$.

Consider any triangle $xyz$ in $G_0$ created by the identification; i.e., $K=x_1x_2x_3yz$ is a $5$-cycle in $G$.
Since $G$ is critical, $x_2$ has degree at least three, and thus $K$ does not bound a face.  By Lemma~\ref{lemma-fr},
$K$ separates a triangle of $G$ from the hole of the disk.  If $K$ separated both triangles, then Theorems~\ref{thm-onetri}
and \ref{thm-3t} would imply that $\psi_3$ extends to a $3$-coloring of $G_0$, which is a contradiction.
Consequently, $K$ separates exactly one of the triangles of $G$ from the hole.
Let $G_1$ be a maximal critical subgraph of $G_0$, and note that by Theorem~\ref{thm-onetri}, $G_1$ contains exactly two triangles.

By the induction hypothesis, $G_1$ is either a tent or
obtained from a patched Havel-Thomas-Walls graph by framing on its interface pair.

Let us first discuss the case that $G_1$ is a tent.  Then $G_1$ contains two vertex-disjoint triangles, each of them sharing an edge with $C$.
At least one of the triangles does not contain $x$, say a triangle $v_1v_2z_1$.  Hence, $v_1v_2z_1$ is a triangle in $G$ as well. By Theorem~\ref{thm-53} applied to the
disk bounded by the $5$-cycle $K=v_1z_1v_2v_3v_4$ and the $3$-coloring of $K$ that extends $\psi_3$,
we conclude that $G$ also contains a triangle containing the edge $v_3v_4$ and all other faces of $G$ have length $4$.
Therefore, $G$ is a tent.

Hence, it remains to consider the case that $G_1$ is obtained from a patched Havel-Thomas-Walls graph by framing on its interface pair, say $v_1v_3$.
Since $\psi_3$ does not extend to a $3$-coloring of $G_1$, Lemma~\ref{lemma-col-htw} implies that $C$ is strong in $G_1$,
and thus $G_1$ is a patched Havel-Thomas-Walls graph.

As the next case, suppose that $G_1$ is not obtained by patching from the graph depicted in Figure~\ref{fig-havel}(b).
Then, since $C$ is a strong ring $G_1$ and since $C$ is an induced cycle in $G_1$, it follows that $G_1$ contains vertices $w_1$, $w_2$, $y_1$ and $y_2$ and facial $5$-cycles $K_1=v_2v_1v_4w_1y_1$ and 
$K_2=v_2v_3v_4w_2y_2$, where possibly $w_1=w_2$.  Furthermore, if $w_1\neq w_2$, then $G_1$ also contains a $6$-cycle $y_1w_1v_4w_2y_2z$ with
quadrangulated interior.  If $w_1=w_2$, let us define $z=w_1$.  Let $K=zy_1v_2y_2$ and let $G_{1,K}$ be the subgraph of $G_1$ drawn in the closed disk bounded
by $K$; note that $G_{1,K}$ is obtained from a patched Havel-Thomas-Walls graph by framing on its interface pair $v_2z$. See Figure~\ref{fig-2triaglesproof}.

\begin{figure}
\begin{center}
\includegraphics[page=1]{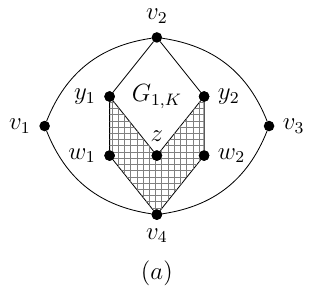}
\hskip 2em
\includegraphics[page=2]{fig-2triaglesproof}
\end{center}
\caption{Graph $G_1$ in proof of Theorem~\ref{thm-crtri}.}\label{fig-2triaglesproof}
\end{figure}

Both $v_1$ and $v_3$ have degree two in $G_1$.  Since $v_1$ and $v_3$ have degree at least three
in $G$ and every non-facial $(\le\!5)$-cycle in $G$ separates the hole from at least one of the triangles of $G$,
it follows that neither $K_1$ nor $K_2$ is a cycle in $G$.  In particular, $x$ is one of $v_4$, $v_2$ or $w_1$ (in the case that $w_1=w_2$).
Furthermore, $K$ corresponds to a $4$-cycle $K'$ in $G$, and the subgraph $G_{K'}$ of $G$ drawn in the closed disk bounded by $K'$ is isomorphic to $G_{1,K}$.
By Lemma~\ref{lemma-col-htw}, any $y_1$-diagonal $3$-coloring of $K'$ extends to a $3$-coloring of $G_{K'}$. Let us distinguish two subcases.
\begin{itemize}
\item If $x=v_2$, then $G$ contains cycles $K'_1=v_2v_1v_4w_1y_1x_1x_2$ and $K'_2=v_2v_3v_4w_2y_2x_1x_2$ (recall that $x_2\not\in V(C)$, and that $v_1,v_3$ have degree at least three in $C$
and non-facial $(\le\!5)$-cycles in $G$ separate the hole from at least one of the triangles,
and thus $x_2\not\in\{v_1,v_3,y_1,y_2\}$).
Let $G_{K'_1}$ and $G_{K'_2}$ denote the subgraphs of $G$ drawn in the closed disks bounded by $K'_1$ and $K'_2$, respectively.

Since $x_2$ has degree at least three in $G$ and every non-facial $(\le\!5)$-cycle in $G$ separates a triangle from the hole,
either $v_1y_1$ or $v_3y_2$ is not an edge.  By symmetry, we can assume the former.
Let $\varphi$ be a $3$-coloring defined by $\varphi(v_1)=\varphi(v_3)=\varphi(x_2)=\varphi(y_1)=1$, $\varphi(v_2)=\varphi(v_4)=\varphi(y_2)=2$ and
$\varphi(x_1)=\varphi(w_1)=\varphi(w_2)=\varphi(z)=3$.  Since $\varphi$ is $y_1$-diagonal on $K'$, it extends to a $3$-coloring of $G_{K'}$.
If $w_1\neq w_2$, then by Lemma~\ref{lemma-split}, $\varphi$ extends to the subgraph of $G$ drawn in the closed disk bounded by $v_4w_1y_1zy_2w_2$.

Suppose that $\varphi$ does not extend to a $3$-coloring of $G_{K'_1}$.  By Theorem~\ref{thm-7cyc}, $G_{K'_1}$ contains a $5$-face
whose intersection with $K'_1$ is a path containing $v_1$, $v_2$, $x_1$, and $y_1$.  However, this is not possible, since $y_1$ is not adjacent to $v_1$.
Hence, $\varphi$ extends to a $3$-coloring of $G_{K'_1}$.

Suppose that $\varphi$ does not extend to a $3$-coloring of $G_{K'_2}$.  By Theorem~\ref{thm-7cyc}, $G_{K'_2}$ contains a $5$-face
whose intersection with $K'_2$ is a path containing $w_2$, $y_2$, $v_2$, and $v_3$.  However, this is not possible, since $v_3$ has degree at least three in $G$.
Hence, $\varphi$ extends to a $3$-coloring of $G_{K'_2}$.

We conclude that $\varphi$ extends to a $3$-coloring of $G$.  This is a contradiction, since $\varphi$ is bichromatic on $C$.
\item If $x\neq v_2$, then $G$ contains cycles $K'_1=v_2v_1v_4abw_1y_1$ and $K'_2=v_2v_3v_4abw_2y_2$, with $\{a,b\}=\{x_1,x_2\}$.
Let $G_{K'_1}$ and $G_{K'_2}$ denote the subgraphs of $G$ drawn in the closed disks bounded by $K'_1$ and $K'_2$, respectively.

Let $\varphi$ be a $3$-coloring defined by
$\varphi(v_1)=\varphi(v_3)=\varphi(a)=\varphi(y_1)=1$,
$\varphi(v_2)=\varphi(v_4)=\varphi(w_1)=\varphi(w_2)=\varphi(z)=2$ and
$\varphi(y_2)=\varphi(b)=3$. Since $\varphi$ is $y_1$-diagonal on $K'$, it extends to a $3$-coloring of $G_{K'}$.
If $w_1\neq w_2$, then by Lemma~\ref{lemma-split}, $\varphi$ extends to the subgraph of $G$ drawn in the closed disk bounded by $bw_1y_1zy_2w_2$.

Suppose that $\varphi$ does not extend to a $3$-coloring of $G_{K'_1}$.  By Theorem~\ref{thm-7cyc}, $G_{K'_1}$ contains a $5$-face
whose intersection with $K'_1$ is a path containing $v_1$, $v_2$, $v_4$, and $y_1$.  However, this is not possible, since $v_1$ has degree at least three in $G$.

Suppose that $\varphi$ does not extend to a $3$-coloring of $G_{K'_2}$.  By Theorem~\ref{thm-7cyc}, $G_{K'_2}$ contains a $5$-face
whose intersection with $K'_2$ is a path containing $v_3$, $v_4$, $w_2$ and $y_2$.  However, this is not possible, since $v_3$ has degree at least three in $G$.

We conclude that $\varphi$ extends to a $3$-coloring of $G$.  This is a contradiction, since $\varphi$ is bichromatic on $C$.
\end{itemize}

Finally, let us consider the case that $G_1$ is obtained by patching from the graph depicted in Figure~\ref{fig-havel}(b).
Since $C$ is strong, $G_1$ contains $5$-faces $v_2v_1v_4w_1y_1$ and $v_2v_3v_4w_2y_2$.  Let $H$ denote the
subgraph of $G_1$ drawn in the closed disk bounded by the $6$-cycle $K=v_2y_1w_1v_4w_2y_2$, and observe that
a precoloring $\varphi$ of $K$ extends to a $3$-coloring of $H$, unless $\{\varphi(y_1),\varphi(w_1)\}=\{\varphi(y_2),\varphi(w_2)\}$.
Since both $v_1$ and $v_3$ have degree at least three in $G$, we conclude that $x\in \{v_2,v_4\}$, say $x=x_3=v_2$,
and $G$ contains a $6$-cycle $K'=x_1y_1w_1v_4w_2y_2$ such that the subgraph drawn in the closed disk bounded by $K'$ is isomorphic to $H$.

Because $x_2$ has degree at least three in $G$ and non-facial $(\le\!5)$-cycles in $G$ separate the hole from at least one of the triangles,
either $v_1y_1$ or $v_3y_2$ is not an edge; assume the latter.
Let $\varphi$ be a $3$-coloring defined by $\varphi(v_2)=\varphi(v_4)=\varphi(y_1)=1$, $\varphi(v_1)=\varphi(v_3)=\varphi(y_2)=\varphi(x_2)=2$
and $\varphi(x_1)=\varphi(w_1)=\varphi(w_2)=3$.  Note that $\varphi$ extends to a $3$-coloring of $H$.
We consider the subgraphs $G_{K'_1}$ and $G_{K'_2}$ of $G$ drawn inside the $7$-cycles $v_2x_2x_1y_1w_1v_4v_1$ and $v_2x_2x_1y_2w_2v_4v_3$, respectively.

Suppose that $\varphi$ does not extend to a $3$-coloring of $G_{K'_1}$.  By Theorem~\ref{thm-7cyc}, $G_{K'_1}$ contains a $5$-face 
whose intersection with $K'_1$ is a path containing $v_1$, $v_2$, $y_1$, and $w_1$.  However, this is not possible, since $v_1$ has degree at least three in $G$.

Suppose that $\varphi$ does not extend to a $3$-coloring of $G_{K'_2}$.  By Theorem~\ref{thm-7cyc}, $G_{K'_2}$ contains a $5$-face 
whose intersection with $K'_2$ is a path containing $v_2$, $v_3$, $x_1$ and $y_2$.  However, this is not possible, since $y_2$ is not adjacent to $v_3$.

Therefore, $\varphi$ extends to a $3$-coloring of $G$.
This is a contradiction, since the restriction of $\varphi$ to $C$ is bichromatic.
\end{proof}

Theorem~\ref{thm-crtri} enables us to give some information about critical graphs embedded in the cylinder with rings of length $4$.

\begin{corollary}\label{cor-xi}
Let $G$ be a critical tame graph embedded in the cylinder with rings $C_1=u_1u_2u_3u_4$ and $C_2=v_1v_2v_3v_4$.
Let $\psi$ be a $3$-coloring of $C_1$.
If no $3$-coloring of $G$ that extends $\psi$ is $v_1$-diagonal on $C_2$, then $G$ is obtained from a patched Thomas-Walls graph by
framing on its interface pairs, one of which is $v_1v_3$.
\end{corollary}
\begin{proof}
Let $C_2'$ be a non-contractible $4$-cycle in $G$ containing $v_1$ and $v_3$ such that the subgraph $G_2$ of $G$ drawn
between $C_2'$ and $C_2$ is maximal.  Let $G_1$ be the subgraph of $G$ drawn between $C_1$ and $C'_2$.
Note that all faces of $C_2'\cup C_2$ have length $4$.  Since $G$ is critical, Lemma~\ref{lemma-fr} implies that
$G_2=C_2'\cup C_2$, and thus $G$ is obtained from $G_1$ by framing on $v_1v_3$.  

Let $G'_1=G_1+v_1v_3$.  Note that $\psi$ does not extend to a $3$-coloring of $G'_1$.   By the choice of $C'_2$,
the edge $v_1v_3$ belongs to exactly two triangles $v_1v'_2v_3$ and $v_1v'_4v_3$ in $G'_1$.  If $G'_1$ contains a triangle $T$ distinct from $v_1v'_2v_3$ and $v_1v'_4v_3$,
then $T$ separates the hole bounded by $C_1$ from $v_1v'_2v_3$ and $v_1v'_4v_3$,
and $\psi$ extends to a $3$-coloring of $G'_1$ by Theorems~\ref{thm-onetri} and \ref{thm-3t}.
This is a contradiction, and thus $G'_1$ contains exactly two triangles.  Since the two triangles of $G'_1$ share an edge,
the examination of the outcomes of Theorem~\ref{thm-crtri} shows that $G'_1$ contains a subgraph $H'$
that is obtained from a patched Thomas-Walls graph by framing on its interface pair in $C_1$, $v_1v_3$ is an interface pair of $H'$,
and the rings of $H'$ are $C_1$ and $C'_2$.

Let $H=H'\cup C_2$.  To prove Corollary~\ref{cor-xi}, it suffices to show that $G=H$.  This is the case, since $H\subseteq G$, every
face of $H$ has length at most $5$, and every contractible $(\le\!5)$-cycle in $G$ bounds a face by Lemma~\ref{lemma-fr}.
\end{proof}

We can now strengthen the conclusions of Lemma~\ref{lemma-domin}.

\begin{lemma}\label{lemma-equiv}
Let $G$ be a tame graph embedded in the cylinder with rings of length at most $4$.  If $G$ is a chain of graphs, at least $264$ of which are not quadrangulated,
then either every precoloring of the rings of $G$ extends to a $3$-coloring of $G$, or $G$ contains
a subgraph obtained from a patched Thomas-Walls graph by framing on its interface pairs, with the same rings as $G$.
\end{lemma}
\begin{proof}
Let $C_1$ and $C_2$ be the rings of $G$.  There exists a non-contractible $(\le\!4)$-cycle $K$ such that for $i\in\{1,2\}$, if $G_i$ denotes the subgraph
of $G$ drawn between $C_i$ and $K$, then $G_i$ is a chain of graphs, at least $132$ of which are not quadrangulated.

Suppose that there exists a $3$-coloring $\psi$ of $C_1\cup C_2$ that does not extend to a $3$-coloring of $G$.
By Corollary~\ref{cor-chaincol}, for $i\in\{1,2\}$ there exists a vertex $v_i\in V(K)$ such that
for any $v_i$-diagonal $3$-coloring $\psi'$ of $K$, the coloring $(\psi\restriction V(C_i))\cup \psi'$ extends to a $3$-coloring of $G_i$.
If $v_1$ is either equal or non-adjacent to $v_2$, then
we can choose a $3$-coloring $\psi'$ of $K$ that is both $v_1$-diagonal and $v_2$-diagonal,
and extend $\psi\cup \psi'$ to both to $G_1$ and $G_2$, which is a contradiction.

Therefore, assume that $v_1$ and $v_2$ are adjacent, $K=v_1v_2v_3v_4$.
Consider any $3$-coloring $\psi'$ of $C_1\cup C_2\cup K$ that extends $\psi$, such that $\psi'$ is $v_2$-diagonal on $K$.
It follows that $\psi'$ extends to a $3$-coloring of $G_2$, and thus it does not extend
to a $3$-coloring of $G_1$.  By Corollary~\ref{cor-xi},  $G_1$ contains a subgraph $H_1$ obtained from a patched Thomas-Walls graph by
framing on its interface pairs, $v_2v_4$ is an interface pair of $H_1$, and the rings of $H_1$ are $C_1$ and $K$.
By symmetry, $G_2$ contains a subgraph $H_2$ obtained from a patched Thomas-Walls graph by
framing on its interface pairs, $v_1v_3$ is an interface pair of $H_2$, and the rings of $H_2$ are $C_2$ and $K$.

Note that all faces of $H_1\cup H_2$ have length at most $5$, and since $\psi$ does not extend to a $3$-coloring of $G$,
Lemma~\ref{lemma-split} implies that $\psi$ does not extend to a $3$-coloring of $H_1\cup H_2$.
If $K$ is weak in both $H_1$ and $H_2$, then we can extend $\psi$ to a $3$-coloring $\psi'$ of $C_1\cup C_2\cup K$ that
is bichromatic on $K$, and further extend $\psi'$ to a $3$-coloring of $H_1$ and $H_2$ by Lemma~\ref{lemma-col-tw},
which is a contradiction.  Hence, $K$ is weak in at most one of $H_1$ and $H_2$.
Let $H$ be the subgraph of $H_1\cup H_2$ obtained by removing all vertices of degree two not belonging to $C_1\cup C_2$.
Observe that $H$ is obtained from a patched Thomas-Walls graph by framing on its interface pairs, as required by the conclusion
of the lemma.
\end{proof}

Let us remark that by using Lemma~\ref{lemma-chaincol} instead of Corollary~\ref{cor-chaincol}, the assumption of Lemma~\ref{lemma-equiv} could
be relaxed to ``If $G$ is a chain of at least $14$ graphs, at least $10$ of which are not quadrangulated''.

\section{Colorings of quadrangulations}\label{sec-qua}

Next, we explore the graphs containing a long chain of quadrangulations, which complements Lemma~\ref{lemma-equiv}.
We need the following fact, which follows from Lemmas 4 and 5 of~\cite{col8cyc}.

\begin{lemma}\label{lemma-ext453}
Let $G$ be a graph embedded in the cylinder with a ring $C=v_1v_2v_3v_4$ and a ring $T$ of length $3$, such that $T$ is the only triangle in $G$,
$G$ has exactly one face $f$ of length $5$, and all non-ring faces of $G$ other than $f$ have length
$4$.  Let $\psi$ be a $3$-coloring of $C$, and let $w\in\{-1,1\}$.  If $\psi$ does not extend to a $3$-coloring of $G$
with winding number $w$ on $T$, then either $T$ shares an edge with $C$, or there exists a path $v_ixyv_{i+2}$ in $G$ for some
$i\in\{1,2\}$ such that $f$ is drawn inside the contractible $5$-cycle of $C+v_ixyv_{i+2}$,
and $\psi(v_i)\neq\psi(v_{i+2})$.
\end{lemma}

Let $G$ be a graph embedded in the cylinder with rings $C$ and $T$ such that $|T|=3$.  Let $\psi$ be a $3$-coloring of $C$.  We say that
\emph{$\psi$ forces the winding number of $T$} if there exists $w\in\{-1,1\}$ such that for every $3$-coloring $\varphi$ of $G$ that
extends $\psi$, the winding number of $\varphi$ on $T$ is $w$.

\begin{lemma}\label{lemma-orfor}
Let $G$ be a critical graph embedded in the cylinder with rings $C$ of length at most $4$ and $T$ of length $3$,
such that all triangles in $G$ are non-contractible.
If there exists a $3$-coloring $\psi$ of $C$ that forces the winding number of $T$, then
$G$ is a near $3,3$-quadrangulation and for some $w\in\{-1,1\}$, $\psi$ on $C$ causes winding number $w$.
\end{lemma}
\begin{proof}
We proceed by induction, assuming that the claim holds for all graphs with less than $|V(G)|$ vertices.
If $C$ and $T$ share at least two vertices, or if $|C|=3$ and $|V(C)\cap V(T)|=1$,
then the claim follows from Lemma~\ref{lemma-split}.  Hence, assume that
$C$ intersects $T$ in at most one vertex, and if $|C|=3$, then $C$ and $T$ are vertex-disjoint.

If $G$ contains a triangle $T'$ distinct from $C$ and $T$, then let $G_1$ be the subgraph of $G$ drawn between $C$ and $T'$, 
and let $G_2$ be the subgraph of $G$ drawn between $T'$ and $T$.  By Theorem~\ref{thm-onetri}, $\psi$ extends to a $3$-coloring $\varphi$ of $G_1$.
Note that $\varphi\restriction V(T')$ must force the winding number of $T$ in $G_2$.  By the induction hypothesis, $G_2$ is a $3,3$-quadrangulation.
By Lemma~\ref{lemma-cylflow}, the winding number of any $3$-coloring of $G_2$ on $T$ is equal to its winding number on $T'$, and we conclude that $\psi$ forces
the winding number of $T'$ in $G_1$.  The claims of Lemma~\ref{lemma-orfor} then follow by the induction hypothesis applied to $G_1$.
Therefore, we can assume that $G$ contains no triangle distinct from the rings, and in particular $G$ is tame.

Suppose that $G$ contains at most one non-ring face $f$ of length other than $4$, and if $G$ has such a face $f$,
that $|f|=5$.  Since $G$ has even number of odd faces, note that $f$ exists if and only if $|C|=4$.
If $|C|=3$, then it follows that $G$ is a $3,3$-quadrangulation, and $\psi$ automatically causes winding number.
If $|C|=4$, then $G$ satisfies the assumptions of
Lemma~\ref{lemma-ext453}.  Since $T$ does not share an edge with $C$, it follows that say $C=v_1v_2v_3v_4$ and $G$ contains
a path $v_1xyv_3$ such that $f$ is contained inside the contractible $5$-cycle
$v_3v_4v_1xy$, and $\psi(v_1)\neq \psi(v_3)$.  Since $G$ is critical, Lemma~\ref{lemma-fr} implies that $f=v_3v_4v_1xy$,
and thus $G$ is a near $3,3$-quadrangulation and $\psi$ on $C$ causes winding number $w$.

\begin{figure}
\begin{center}
\includegraphics[scale=0.7]{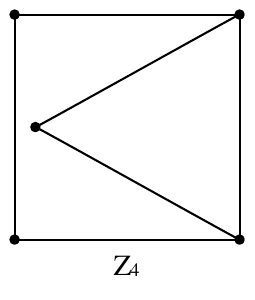}
\includegraphics[scale=0.7]{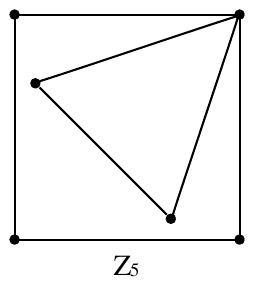}
\includegraphics[scale=0.7]{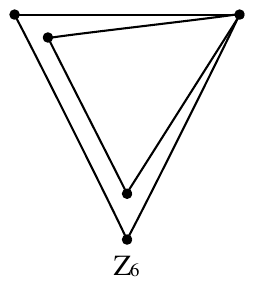}
\includegraphics[scale=0.7]{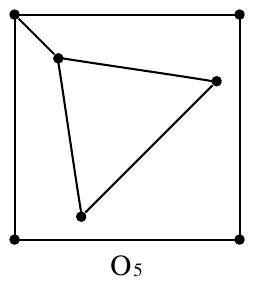}
\includegraphics[scale=0.7]{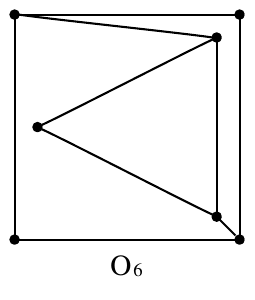}
\includegraphics[scale=0.7]{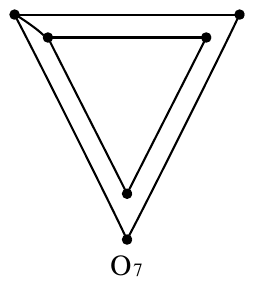}
\includegraphics[scale=0.7]{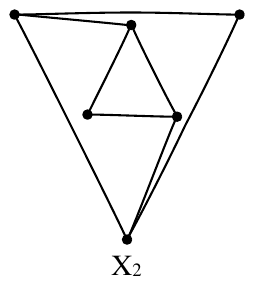}
\includegraphics[scale=0.7]{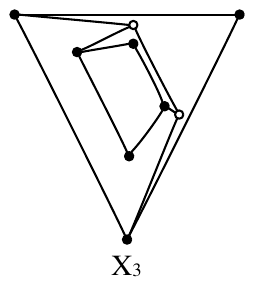}
\includegraphics[scale=0.7]{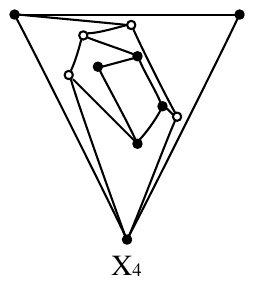}
\includegraphics[scale=0.7]{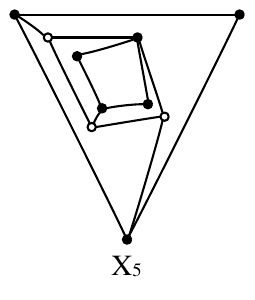}
\includegraphics[scale=0.7]{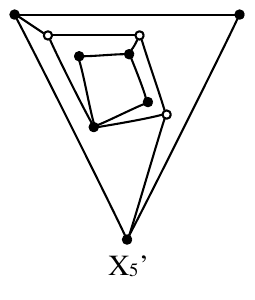}
\includegraphics[scale=0.7]{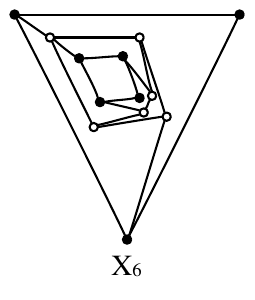}
\includegraphics[scale=0.7]{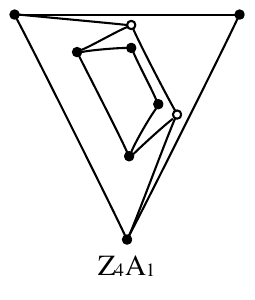}
\includegraphics[scale=0.7]{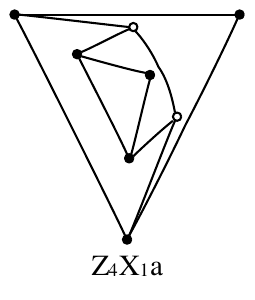}
\end{center}
\caption{Tame critical graphs with rings of length $3$ and at most $4$, no contractible $4$-cycles, and no non-ring triangles.}\label{fig-cr43}
\end{figure}

Hence, assume that
\claim{cl-moref}{$G$ contains either at least two faces of length at least $5$, or a face of length at least $6$.}
If $G$ contains no $4$-face, then $G$ is one of the critical graphs determined in~\cite{dvolid}.
We depict those without separating triangles in Figure~\ref{fig-cr43}.
A straightforward case analysis shows that $\psi$ does not force the the winding number of $T$ in any of these graphs.

Therefore, we can assume that $G$ contains a $4$-face $f=u_1u_2u_3u_4$.
Suppose first that three vertices of $f$, say $u_1$, $u_2$, and $u_3$, either all belong to $C$, or all belong to $T$.
Since $G$ is tame, this is only possible if $u_1,u_2,u_3\in V(C)$ and $|C|=4$.
Let $C=u_1u_2u_3v_4$.  Let $G'=G-u_2$, and let $\psi'$ be
the $3$-coloring of $C'=u_1u_4u_3v_4$ given by $\psi'(x)=\psi(x)$ for $x\in\{u_1,u_3,v_4\}$ and $\psi'(u_4)=\psi(v_4)$.
Note that $\psi'$ forces the winding number of $T$ in $G'$, and by the induction hypothesis, $G'$ is a near $3,3$-quadrangulation.
Since $f$ is the only face of $G$ that does not belong to $G'$, this contradicts (\ref{cl-moref}).
It follows that we can assume that $|V(f)\cap V(C)|\le 2$ and $|V(f)\cap V(T)|\le 2$ for every $4$-face $f$ of $G$.

In particular, we can by symmetry assume that $|\{u_1,u_3\}\cap V(C)|\le 1$ and $|\{u_1,u_3\}\cap V(T)|\le 1$; hence, $u_1$ and $u_3$ are not both contained
in the same triangle, and thus they are non-adjacent.
Let $G_1$ be obtained from $G$ by identifying $u_1$ with $u_3$, and let $G_2$ be a maximal critical subgraph of $G_1$.

Suppose for a contradiction that $G_2$ contains a contractible triangle, and thus $G$ contains a contractible $5$-cycle $K$ with $u_1u_2u_3\subset K$.
By Lemma~\ref{lemma-fr}, $K$ bounds a face in $G$, and thus $u_2$ has degree two.  Since $G$ is critical, we conclude
that $u_2$ is incident with $C$ or $T$.  However, then $u_2$ and its neighbors $u_1$ and $u_3$ all belong to $C$ or all belong to $T$,
which is a contradiction.

Hence, every triangle in $G_2$ is non-contractible.
Note that $\psi$ forces the winding number of $T$ in $G_2$, since every precoloring of $C\cup T$ that extends to a $3$-coloring of $G_2$ also
extends to a $3$-coloring of $G$.  By the induction hypothesis, $G_2$ is a near $3,3$-quadrangulation.
Each $(\le\!5)$-face $K$ in $G_2$ corresponds either to a $|K|$-face in $G$,
or to a contractible $(|K|+2)$-cycle $K'$ in $G$ containing either the path $u_1u_2u_3$ or the path $u_1u_4u_3$.  Since neither $u_2$ nor $u_4$ has degree $2$
in $G$, in the latter case $K'$ does not bound a face, and by Theorems~\ref{thm-6cyc} and~\ref{thm-7cyc}, all the faces contained in the disk bounded by $K'$
have length $4$ except for one of length $|K|$.  We conclude that all faces of $G$ distinct from $C$ and $T$ have length four,
except possibly for one of length $5$.  This contradicts (\ref{cl-moref}).
\end{proof}

As a corollary, we obtain the following.

\begin{lemma}\label{lemma-near33}
Let $G$ be a tame graph embedded in the cylinder with rings $C_1$ and $C_2$ of length at most $4$.
Suppose that $G$ contains non-contractible $(\le\!4)$-cycles $K_1$ and $K_2$ at distance at least $4$ from each other,
such that all faces of $G$ drawn between $K_1$ and $K_2$ have length $4$.  Then either every precoloring of $C_1\cup C_2$
extends to a $3$-coloring of $G$, or $G$ contains a near $3,3$-quadrangulation with rings $C_1$ and $C_2$
as a subgraph.
\end{lemma}
\begin{proof}
Without loss of generality, $K_1$ separates $C_1$ from $K_2$.  For $i\in\{1,2\}$, let $G_i$ be the subgraph of $G$ drawn
between $C_i$ and $K_i$.  Let $G_0$ be the subgraph of $G$ drawn between $K_1$ and $K_2$.

Suppose that there exists a precoloring $\psi$ of $C_1\cup C_2$ that does not extend to a $3$-coloring of $G$.
By Theorem~\ref{thm-onetri}, $\psi$ extends to a $3$-coloring $\varphi$ of $G_1\cup G_2$.  Since $\psi$ does not extend
to a $3$-coloring of $G$, $\varphi\restriction V(K_1\cup K_2)$ does not extend to a $3$-coloring of $G_0$.
By Lemma~\ref{lemma-col44}, it follows that $|K_1|=|K_2|=3$ and $\varphi$ has opposite winding numbers on $K_1$ and $K_2$.
Furthermore, $\psi$ forces the winding number of $K_1$ and $K_2$, and thus by Lemma~\ref{lemma-orfor}, for $i\in\{1,2\}$,
$G_i$ contains a near $3,3$-quadrangulation $H_i$ with rings $C_i$ and $K_i$ as a subgraph.
Then, $H_1\cup G_0\cup H_2$ is a near $3,3$-quadrangulation with rings $C_1\cup C_2$.
\end{proof}

\section{Chains in cylinder}\label{sec-many}

We can now prove the main result of the paper.

\begin{proof}[Proof of Theorem~\ref{thm-many}]
Let $c_1=264$ be the constant of Lemma~\ref{lemma-equiv}.  Let $c=4c_1=1056$.  If at least $c_1$ of the graphs forming the chain $G$ are not quadrangulations,
then by Lemma~\ref{lemma-equiv}, either every precoloring of $C_1\cup C_2$ extends to a $3$-coloring of $G$, or
$G$ contains a subgraph $H$ obtained from a patched Thomas-Walls graph by framing on its interface pairs, and the
rings of $H$ are $C_1$ and $C_2$.

On the other hand, if all but at most $c_1-1$ graphs in the chain forming $G$ are quadrangulations, then
there exist four consecutive graphs in the chain that are quadrangulations, and thus by Lemma~\ref{lemma-near33},
either every precoloring of $C_1\cup C_2$ extends to a $3$-coloring of $G$, or
$G$ contains a near $3,3$-quadrangulation with rings $C_1$ and $C_2$ as a subgraph.
\end{proof}

Using the bound from Lemma~\ref{lemma-chaincol}, the constant $c$ of Theorem~\ref{thm-many} can be improved to $40$.  However, even this improved
bound is still likely to be far from the best possible.

\bibliographystyle{acm}
\bibliography{cylgen}

\end{document}